\documentclass[12pt,pdftex,noinfoline]{imsart}

\RequirePackage[OT1]{fontenc}
\usepackage[usenames,dvipsnames]{xcolor}
\usepackage[text={6.0in,8.6in},centering]{geometry}
\usepackage{amsthm,amsmath,amsfonts,amssymb}
\usepackage[numbers]{natbib}
\RequirePackage{hypernat}
\usepackage{enumerate}
\usepackage{graphicx}
\usepackage{float}
\usepackage{tikz}
\usepackage{fullpage}

\usepackage{hyperref}
\definecolor{shadecolor}{gray}{0.9}
\hypersetup{citecolor=MidnightBlue}
\hypersetup{linkcolor=MidnightBlue}
\hypersetup{urlcolor=MidnightBlue}

\newcommand{\E}{\mathbb{E}\,}
\newcommand{\F}{{\mathcal F}}

\renewcommand{\tilde}{\widetilde}
\renewcommand{\hat}{\widehat}
\newcommand{\eps}{\varepsilon}

\DeclareMathOperator*{\argmin}{\arg\min}
\DeclareMathOperator*{\argmax}{\arg\max}

\numberwithin{equation}{section}
\newtheorem{theorem}{Theorem}[section]
\newtheorem{lemma}{Lemma}[section]
\newtheorem{corollary}{Corollary}[section]

\theoremstyle{remark}
\newtheorem{example}{Example}[section]
\newtheorem{remark}{Remark}[section]

\begin{document}

\begin{frontmatter}

\title{Quantized Minimax Estimation\\ over Sobolev Ellipsoids} 
\runtitle{}
\runauthor{}

\begin{aug}
\vskip10pt
\author{\fnms{Yuancheng}
  \snm{Zhu${}^{*}$}\ead[label=e1]{zhuyuanc@wharton.upenn.edu}}
\, \and \,
\author{\fnms{John}
  \snm{Lafferty${}^{\dag}$}\ead[label=e2]{lafferty@galton.uchicago.edu}}
\vskip10pt
\address{
\begin{tabular}{cc}
${}^*$Department of Statistics & ${}^\dag$Department of Statistics\\
The Wharton School & Department of Computer Science\\
University of Pennsylvania & University of Chicago
\end{tabular}
\\[10pt]
\today\\[5pt]
\vskip10pt
}
\end{aug}

\begin{abstract}
{We formulate the notion of
minimax estimation under storage or communication constraints,
and prove an extension to Pinsker's theorem
for nonparametric estimation over Sobolev ellipsoids.
Placing limits on the number of bits
used to encode any estimator, we give tight lower and upper bounds
on the excess risk due to quantization in terms of
the number of bits, the signal size, and the noise level.  This
establishes the Pareto optimal tradeoff between storage and
risk under quantization constraints for Sobolev spaces.  
Our results and proof techniques combine elements of rate distortion theory
and minimax analysis. The proposed quantized estimation scheme, which shows achievability of the
lower bounds, is adaptive in the usual statistical sense, achieving
the optimal quantized minimax rate without knowledge of the smoothness
parameter of the Sobolev space.  It is also adaptive
in a computational sense, as it constructs the code only after
observing the data, to dynamically allocate more codewords to blocks
where the estimated signal size is large. Simulations are included
that illustrate the effect of quantization on statistical risk.}
{nonparametric estimation, minimax bounds, rate distortion theory, constrained estimation, Sobolev ellipsoid}
\end{abstract}
 
\vskip20pt 
\end{frontmatter}

\maketitle

\vskip10pt
\section{Introduction}

In this paper we introduce a minimax framework for nonparametric
estimation under storage constraints.  In the classical statistical
setting, the minimax risk  
for estimating a function $f$ from a function class $\F$
using a sample of size $n$ places no constraints on the estimator
${\hat f_n}$, other than requiring it 
to be a measurable function of the data.  However, if
the estimator is to be constructed with restrictions
on the computational resources used,
it is of interest to understand how the error can degrade.
Letting $C(\hat f_n) \leq B_n$ indicate that the
computational resources $C(\hat f_n)$ used to construct $\hat f_n$ are
required to fall within a budget $B_n$, the constrained minimax risk is
$$ R_n(\F,B_n) = \inf_{\hat f_n: C(\hat f_n)\leq B_n} \sup_{f\in \F}
R(\hat f_n, f).$$ Minimax lower bounds on the
risk as a function of the computational budget thus determine a
feasible region for computation constrained estimation, and a Pareto
optimal tradeoff for risk versus computation as $B_n$ varies.  

Several recent papers have presented results on tradeoffs between statistical
risk and computational resources, measured in terms of either running time of the algorithm,
number of floating point operations, or number of bits used to store or construct
the estimators \cite{chandrasekaran2013computational, bruer2014time,
lucic2015tradeoffs}. However, the existing work quantifies the
tradeoff by analyzing the statistical and computational performance
of specific procedures, rather than by establishing lower bounds and
a Pareto optimal tradeoff. In this paper we treat
the case where the complexity $C(\hat f_n)$ is measured by the
storage or space used by the procedure and sharply characterize the optimal tradeoff.
Specifically, we limit the
number of bits used to represent the estimator $\hat f_n$.  We focus
on the setting of nonparametric regression under standard smoothness
assumptions, and study how the excess risk depends on the storage
budget $B_n$.

We view the study of quantized estimation as a theoretical problem of
fundamental interest. But quantization may arise naturally in future
applications of large scale statistical estimation.  For instance,
when data are collected and analyzed on board a remote satellite, the
estimated values may need to be sent back to Earth for further
analysis. To limit communication costs, the estimates can be
quantized, and it becomes important to understand what, in principle,
is lost in terms of statistical risk through quantization.  A related
scenario is a cloud computing environment where data are processed for
many different statistical estimation problems, with the estimates
then stored for future analysis. To limit the storage costs, which
could dominate the compute costs in many scenarios, it is of interest
to quantize the estimates, and the quantization-risk tradeoff again
becomes an important concern.  Estimates
are always quantized to some degree in practice.  But to impose energy
constraints on computation, future processors may limit precision in
arithmetic computations more significantly \cite{galal2011energy}; the cost
of limited precision in terms of statistical risk must then be
quantified.
A related problem is to distribute the
estimation over many parallel processors, and to then limit the
communication costs of the submodels to the central host.  
We focus on the centralized setting in the current paper,
but an extension to the distributed case may be possible
with the techniques that we introduce here.

We study risk-storage tradeoffs in the
normal means model of nonparametric estimation assuming the target
function lies in a Sobolev space.  The problem is intimately related
to classical rate distortion theory \cite{gallager1968information}, and our
results rely on a marriage of minimax theory and rate distortion
ideas.  We thus build on and refine the connection between
function estimation and lossy source coding that was elucidated in
David Donoho's 1997 Wald Lectures \cite{donoho1997wald}.

We work in the Gaussian white noise model
\begin{equation}
dX(t)=f(t)dt+\varepsilon dW(t), \quad 0\leq t\leq 1,\label{eqn:wnm}
\end{equation}
where $W$ is a standard Wiener process on $[0,1]$, 
$\varepsilon$ is the standard deviation of the noise,
and $f$ lies in the periodic Sobolev space $\tilde W(m,c)$ of order $m$ and radius
$c$. (We discuss the nonperiodic Sobolev space $W(m,c)$ in
Section~\ref{sec:achievability}.) 
The white noise model is a centerpiece of nonparametric estimation.
It is asymptotically equivalent to nonparametric regression \cite{brown1996asymptotic}
and density estimation \cite{nussbaum1996asymptotic}, and simplifies
some of the mathematical analysis in our framework.
In this classical setting,
the minimax risk of estimation 
\begin{equation*}
R_\varepsilon(m,c)=\inf_{\hat f_\varepsilon}\sup_{f\in\tilde W(m,c)}\E\|f-\hat f_\varepsilon\|_2^2
\end{equation*}
is well known to satisfy
\begin{equation}
\lim_{\varepsilon\to 0}\varepsilon^{-\frac{4m}{2m+1}}R_\varepsilon(m,c)=
\left(\frac{c^2(2m+1)}{\pi^{2m}}\right)^{\frac{1}{2m+1}}\left(\frac{m}{m+1}\right)^{\frac{2m}{2m+1}}
\triangleq \mathsf P_{m,c}\label{pinsker}
\end{equation}
where $\mathsf P_{m,c}$ is Pinsker's constant \cite{nussbaum1999minimax}.  The
constrained minimax risk for quantized estimation becomes
\begin{equation*}\label{eqn_minimaxdef}
R_\varepsilon(m,c,B_\eps)=\inf_{\hat f_\varepsilon, C(\hat
f_\varepsilon)\leq B_\eps}\sup_{f\in\tilde W(m,c)}\E\|f-\hat f_\varepsilon\|_2^2
\end{equation*}
where $\hat f_\eps$ is a \textit{quantized estimator} that is
required to use storage $C(\hat f_\eps)$ no greater than $B_\eps$ bits in
total. Our main result identifies three separate quantization
regimes. 

\begin{itemize}
\vskip10pt
\item In
the \textit{over-sufficient regime}, the number of bits
is very large, satisfying $B_\eps \gg \eps^{-\frac{2}{2m+1}}$ and the classical minimax rate
of convergence $R_\eps \asymp \eps^{\frac{4m}{2m+1}}$ is obtained.
Moreover, the optimal constant is the Pinsker constant ${\mathsf
P}_{m,c}$.

\vskip10pt
\item In the \textit{sufficient regime}, the number
of bits scales as $B_\eps \asymp \eps^{-\frac{2}{2m+1}}$.  This level
of quantization is just sufficient to preserve the classical minimax
rate of convergence, and thus in this regime
$R_\varepsilon(m,c,B_\eps) \asymp \eps^{\frac{4m}{2m+1}}$.
However, the optimal constant degrades to a new constant ${\mathsf
P}_{m,c} + {\mathsf Q}_{m,c,d}$, where ${\mathsf Q}_{m,c,d}$ is 
characterized in terms of the solution of a certain variational
problem, depending on $d=\lim_{\eps\rightarrow 0}
B_{\eps} \eps^{\frac{2}{2m+1}}$.

\vskip10pt
\item In the \textit{insufficient regime}, the number
of bits scales as $B_\eps \ll \eps^{-\frac{2}{2m+1}}$, with however
$B_\eps \to \infty$.  Under this scaling
the number of bits is insufficient to preserve the unquantized
minimax rate of convergence, and the quantization error dominates
the estimation error.   We show that the quantized
minimax risk in this case satisfies
\begin{equation*}
\lim_{\varepsilon\to 0} B_\eps^{2m} R_\varepsilon(m,c,B_\eps) = 
\frac{c^2 m^{2m}}{\pi^{2m}}.
\end{equation*}
Thus, in the insufficient regime the
quantized minimax rate of convergence is $B_\eps ^{-2m}$,
with optimal constant as shown above.
\end{itemize}

By using an upper bound for the family of constants ${\mathsf Q}_{m,c,d}$,
the three regimes can be combined together to view 
the risk in terms of a decomposition into estimation error and quantization error.
Specifically, we can write
\begin{equation*}
R_\varepsilon(m,c,B_\varepsilon)\;\; \approx \underbrace{\mathsf P_{m,c}\,\varepsilon^{\frac{4m}{2m+1}}}_{\mbox{\footnotesize estimation error}}
+ \underbrace{\frac{c^2m^{2m}}{\pi^{2m}}B_\varepsilon^{-2m}}_{\mbox{\footnotesize quantization error}}.
\end{equation*}
When $B_\varepsilon\gg\varepsilon^{-\frac{2}{2m+1}}$, the estimation error dominates the quantization error, 
and the usual minimax rate and constant are obtained.
In the insufficient case $B_\varepsilon\ll\varepsilon^{-\frac{2}{2m+1}}$,
only a slower rate of convergence is achievable. When $B_\varepsilon$ and $\varepsilon^{-\frac{2}{2m+1}}$ are comparable, 
the estimation error and quantization error are on the same order.
The threshold $\varepsilon^{-\frac{2}{2m+1}}$ should not be surprising,
given that in classical unquantized estimation the minimax rate
of convergence is achieved by estimating the first $\varepsilon^{-\frac{2}{2m+1}}$ Fourier
coefficients and simply setting the remaining coefficients to zero.
This corresponds to selecting a smoothing bandwidth that scales as
$h\asymp n^{-\frac{1}{2m+1}}$ with the sample size $n$.


At a high level, our proof strategy integrates elements of minimax
theory and source coding theory.  In minimax analysis one computes
lower bounds by thinking in Bayesian terms to look for least-favorable
priors.  In source coding analysis one constructs worst case
distributions by setting up an optimization problem based on mutual
information.  Our quantized minimax analysis requires that these
approaches be carefully combined to balance the estimation and
quantization errors. To show achievability of the lower bounds we
establish, we likewise need to construct an estimator and coding
scheme together.  Our approach is to quantize the blockwise
James-Stein estimator, which achieves the classical Pinsker bound.
However, our quantization scheme differs from the approach taken in
classical rate distortion theory, where the generation of the codebook
is determined once the source distribution is known.  In our setting,
we require the allocation of bits to be adaptive to the data, using
more bits for blocks that have larger signal size.  We therefore
design a quantized estimation procedure that adaptively distributes
the communication budget across the blocks.  Assuming only 
a lower bound $m_0$ on the smoothness $m$ and an upper bound $c_0$ on the radius $c$ of the Sobolev
space, our quantization-estimation procedure is
adaptive to $m$ and $c$ in the usual statistical sense,
and is also adaptive to the coding regime.  In other words, 
given a storage budget $B_\eps$, the coding procedure
achieves the optimal rate and constant for the unknown
$m$ and $c$, operating in the corresponding regime for those
parameters.

In the following section we establish some notation, outline our proof
strategy, and present some simple examples.  In
Section~\ref{sec:lowerbound} we state and prove our main result on
quantized minimax lower bounds, relegating some of the technical
details to an appendix. In Section~\ref{sec:achievability} we show
asymptotic achievability of these lower bounds, using a quantized
estimation procedure based on adaptive James-Stein estimation and
quantization in blocks, again deferring proofs of technical lemmas to
the supplementary material. This is followed by a presentation of some
results from experiments in Section~\ref{sec:experiments},
illustrating the performance and properties of the proposed quantized
estimation procedure.

\section{Quantized estimation and minimax risk} 
\label{sec:general}
\def\given{\,|\,}

Suppose that $(X_1,\dots, X_n)\in\mathcal X^n$ is
a random vector drawn from a distribution $P_n$. Consider the problem of estimating a functional
$\theta_n=\theta(P_n)$ of the distribution, assuming $\theta_n$ is
restricted to lie in a parameter space $\Theta_n$.   To unclutter some
of the notation, we will suppress the subscript $n$ and write $\theta$
and $\Theta$ in the following, keeping in mind that nonparametric settings
are allowed.  The subscript $n$ will be maintained for random variables.
The minimax $\ell_2$ risk 
of estimating $\theta$ is then defined as
\begin{equation*}
R_n(\Theta)=\inf_{\hat\theta_n}\sup_{\theta\in\Theta}\mathbb E_\theta\|\theta-\hat\theta_n\|^2
\end{equation*}
where the infimum is taken over all possible estimators
$\hat\theta_n:\mathcal X^n\to\Theta$ that are measurable
with respect to the data $X_1,\dots,X_n$. 
We will abuse notation by using $\hat\theta_n$ to denote both the estimator 
and the estimate calculated based on an observed set of data.
Among numerous approaches to
obtaining the minimax risk,
the Bayesian method is best aligned with quantized estimation.  
Consider a prior distribution $\pi(\theta)$ whose support is a subset of $\Theta$. Let $\delta(X_{1:n})$
be the posterior mean of $\theta$ given the data $X_1,\dots,X_n$, which minimizes the integrated
risk. Then for any estimator $\hat\theta_n$, 
\begin{equation*}
\sup_{\theta\in\Theta}\mathbb E_\theta\|\theta-\hat\theta_n\|^2\geq\int_\Theta\mathbb E_\theta\|\theta-\hat\theta_n\|^2 d\pi(\theta)\geq  \int_\Theta\mathbb E_\theta\|\theta-\delta(X_{1:n})\|^2 d\pi(\theta).
\end{equation*}
Taking the infimum over $\hat\theta_n$ yields
\begin{equation*}
\inf_{\hat\theta_n}\sup_{\theta\in\Theta}\mathbb
E_\theta\|\theta-\hat\theta_n\|^2\geq\int_\Theta\mathbb
E_\theta\|\theta-\delta(X_{1:n})\|^2 d\pi(\theta)\triangleq
R_n(\Theta;\pi).
\end{equation*}  
Thus, any prior distribution supported on $\Theta$
gives a lower bound on the minimax risk, and selecting the
least-favorable prior leads to the largest lower bound provable by
this approach.

Now consider constraints on the storage or
communication cost of our estimate.
We restrict to the set of estimators that
use no more than a total of $B_n$ bits; 
that is, the estimator takes
at most $2^{B_n}$ different values. 
Such \textit{quantized estimators} can be formulated by the following two-step procedure.
First, an \emph{encoder} maps 
the data $X_{1:n}$ to an index $\phi_n(X_{1:n})$, where
\begin{equation*}
\phi_n:\mathbb {\mathcal X}^n\to\{1,2,\dots,2^{B_n}\}
\end{equation*}
is the \emph{encoding function}.
The \emph{decoder}, after receiving or retrieving the index, 
represents the estimates based on a \emph{decoding function}
\begin{equation*}
\psi_n:\{1,2,\dots,2^{B_n}\}\to \Theta,
\end{equation*}
mapping the index to a codebook of estimates.  All that needs to be
transmitted or stored is the $B_n$-bit-long index, and the quantized
estimator $\hat\theta_n$ is simply $\psi_n\circ\phi_n$, the
composition of the encoder and the decoder functions.  Denoting by
$C(\hat\theta_n)$ the storage, in terms of the number of bits,
required by an estimator $\hat\theta_n$, the minimax risk of quantized
estimation is then defined as
\begin{equation*}
R_n(\Theta, B_n)=\inf_{\hat\theta_n,C(\hat\theta_n)\leq B_n}\sup_{\theta\in\Theta}\mathbb E_\theta\|\theta-\hat\theta_n\|^2,
\end{equation*}
and we are interested in the effect of the constraint on the minimax risk. 
Once again, we consider a prior distribution $\pi(\theta)$ supported
on $\Theta$ and let $\delta(X_{1:n})$ be the posterior mean of
$\theta$ given the data. The integrated risk can then be decomposed as
\begin{equation}
\label{decomp1}\begin{aligned}
\int_\Theta\mathbb E_\theta\|\theta-\hat\theta_n\|^2d\pi(\theta)&=\mathbb E\|\theta-\delta(X_{1:n})+\delta(X_{1:n})-\hat\theta_n\|^2\\
&=\mathbb E\|\theta-\delta(X_{1:n})\|^2+\mathbb E\|\delta(X_{1:n})-\hat\theta_n\|^2
\end{aligned}\end{equation}
where the expectation is with respect to the joint distribution of
$\theta\sim\pi(\theta)$ and $X_{1:n} \given \theta\sim P_\theta$, and the second equality is due to 
\begin{align*}
\nonumber
&\mathbb
E\langle\theta-\delta(X_{1:n}),\delta(X_{1:n})-\hat\theta_n\rangle\\
&=\mathbb E\left(\mathbb E\left(\langle\theta-\delta(X_{1:n}),\delta(X_{1:n})-\hat\theta_n\rangle \given X_{1:n}\right)\right)\\
&=\mathbb E\left(\langle\mathbb E(\theta-\delta(X_{1:n}) \given X_{1:n}),\delta(X_{1:n})-\hat\theta_n \rangle\right)\\
&=\mathbb E\left(\langle0,\delta(X_{1:n})-\hat\theta_n\rangle\right)=0,
\end{align*}
using the fact that $\theta\to X_{1:n}\to \hat\theta_n$ forms
a Markov chain. The first term in the decomposition (\ref{decomp1}) is
the Bayes risk $R_n(\Theta;\pi)$. 
The second term can be viewed as the excess risk due to quantization. 

Let $T_n=T(X_1,\dots,X_n)$ be a sufficient statistic for $\theta$. The
posterior mean can be expressed in terms of $T_n$ and we will abuse 
notation and write it as $\delta(T_n)$. Since the quantized
estimator $\hat\theta_n$ uses at most $B_n$ bits, we have
\begin{equation*}
B_n\geq H(\hat\theta_n)\geq H(\hat\theta_n)-H(\hat\theta_n
\given \delta(T_n))=I(\hat\theta_n;\delta(T_n)),
\end{equation*}
where $H$ and $I$ denote the Shannon entropy and mutual information, respectively.
Now consider the optimization
\begin{align*}
\inf_{P(\cdot \given \delta(T_n))}\ &\ \mathbb E\|\delta(T_n)-\tilde\theta_n\|^2\\
\text{such that}\ &\ I(\tilde\theta_n;\delta(T_n))\leq B_n
\end{align*}
where the infimum is over all conditional distributions
$P(\tilde\theta_n \given \delta(T_n))$. 
This parallels the definition of the distortion rate function,
minimizing the distortion under a constraint on mutual information
\cite{gallager1968information}. 
Denoting the value of this optimization by $Q_n(\Theta, B_n;\pi)$, we can lower bound
the quantized minimax risk by
\begin{equation*}
R_n(\Theta, B_n)\geq R_n(\Theta;\pi)+Q_n(\Theta, B_n;\pi).
\end{equation*}
Since each prior distribution $\pi(\theta)$ supported on $\Theta$ gives a
lower bound, we have
\begin{equation*}\label{eqn_generallowerbound}
R_n(\Theta, B_n)\geq \sup_\pi \Bigl\{R_n(\Theta;\pi)+Q_n(\Theta, B_n;\pi)\Bigr\}
\end{equation*}
and the goal becomes to obtain a least favorable prior for the
quantized risk.

Before turning to the case of quantized estimation
over Sobolev spaces, we illustrate this technique on some simpler,
more concrete examples.

\begin{example}[Normal means in a hypercube]
Let $X_i\sim \mathcal N(\theta,\sigma^2I_d)$ for $i=1,2,\dots,n$. Suppose that $\sigma^2$ is known and $\theta\in[-\tau,\tau]^d$ is to be estimated. We choose the prior $\pi(\theta)$ on $\theta$ to be a product distribution with density
\begin{equation*}
\pi(\theta)=\prod_{j=1}^d \frac{3}{2\tau^3}{\left(\tau-|\theta_j|\right)_+}^2.
\end{equation*}
It is shown in \cite{johnstone2015gaussian} that 
\begin{equation*}
R_n(\Theta;\pi)\geq \frac{\sigma^2
  d}{n}\frac{\tau^2}{\tau^2+12\sigma^2/n}\geq
c_1\frac{\sigma^2 d}{n}
\end{equation*}
where $c_1=\frac{\tau^2}{\tau^2+12\sigma^2}$. Turning to
$Q_n(\Theta,B_n;\pi)$, let $T^{(n)}=(T^{(n)}_1,\dots,T^{(n)}_d)=\mathbb E(\theta|X_{1:n})$ be the 
posterior mean of $\theta$. In fact, by the independence and symmetry among
the dimensions, we know $T_1,\dots,T_d$ are independently and identically 
distributed. Denoting by $T_0^{(n)}$ this common distribution, we have 
\begin{equation*}
Q_n(\Theta,B_n;\pi)\geq d\cdot q(B_n/d)
\end{equation*}
where $q(B)$ is the distortion rate function for $T^{(n)}_0$, i.e.,
the value of the following problem
\begin{align*}
\inf_{P(\hat T\given T_0^{(n)})}\ &\ \mathbb E(T_0^{(n)}-\hat T)^2\\
\text{such that}\ &\ I(\hat T;T_0^{(n)})\leq B.
\end{align*}
Now using the Shannon lower bound \cite{cover2006elements}, we get 
\begin{equation*}
Q_n(\Theta,B_n;\pi)\geq  \frac{d}{2\pi e}\cdot 2^{h(T_0^{(n)})}\cdot 2^{-\frac{2B_n}{d}}.
\end{equation*}
Note that as $n\to\infty$, $T_0^{(n)}$ converges to $\theta$ in distribution, so there exists 
a constant $c_2$ independent of $n$ and $d$ such that
\begin{equation*}
R_n(\Theta,B_n)\geq c_1\frac{\sigma^2 d}{n}+c_2d\,2^{-\frac{2B_n}{d}}.
\end{equation*}
This lower bound intuitively shows the risk is regulated by two
factors, the estimation error and the quantization error; whichever is
larger dominates the risk.   The scaling behavior of this lower bound
(ignoring constants) can be achieved by first quantizing each of the $d$ intervals
$[-\tau,\tau]$ using $B_n/d$ bits each,
and then mapping the {\sc mle} to its closest codeword. 
\end{example}


\begin{example}[Gaussian sequences in Euclidean balls] 
In the example shown above, the lower bound is tight only in terms of
the scaling of the key parameters.  In some instances, we are able to find
an asymptotically tight lower bound for which we can show
achievability of both the rate and the constants. Estimating the mean vector of a Gaussian
sequence with an $\ell_2$ norm constraint on the mean is one of such
case, as we showed in previous work \cite{zhu2014quantized}.

Specifically, let $X_i \sim \mathcal N(\theta_i, \sigma^2_n)$ for $i=1,2,\ldots,
n$, where $\sigma^2_n = \sigma^2/n$.   Suppose that the parameter
$\theta = (\theta_1,\ldots, \theta_n)$ lies in the 
Euclidean ball $\Theta_n(c) = \left\{\theta : \sum_{i=1}^n \theta_i^2 \leq
c^2\right\}$.  Furthermore, suppose that $B_n = nB$.  Then using the prior $\theta_i \sim \mathcal N(0,c^2)$ it can be
shown that 
\begin{equation*}
\liminf_{n\to\infty} R_n(\Theta_n(c), B_n) \geq \frac{\sigma^2
  c^2}{\sigma^2 + c^2} + \frac{c^4 2^{-2B}}{\sigma^2 + c^2} .
\end{equation*}
The asymptotic estimation error
$\sigma^2 c^2 / (\sigma^2 + c^2)$ is the well-known Pinsker
bound for the Euclidean ball case.   
As shown in \cite{zhu2014quantized}, an explicit quantization scheme can be constructed that asymptotically
achieves this lower bound, realizing the smallest possible
quantization error $c^4 2^{-2B} / (\sigma^2 + c^2)$ for a budget of
$B_n = nB$ bits.

The Euclidean ball case is clearly relevant to the Sobolev ellipsoid
case, but new coding strategies
and proof techniques are required.  In particular, as will
be made clear in the sequel, we will use an adaptive allocation of bits
across blocks of coefficients, using more bits for blocks that have
larger estimated signal size.  Moreover, determination of the optimal
constants requires a detailed analysis of the worst case prior
distributions and the solution of a series of variational problems.

\end{example} 

\section{Quantized estimation over Sobolev spaces} \label{sec:lowerbound}
\def\F{{\mathcal F}}
\def\cE{{\mathcal E}}
\def\cQ{{\mathcal Q}}

Recall that the \emph{Sobolev space of order $m$ and radius $c$} is defined by
\begin{align*}
W(m,c)=\Big\{&f\in [0,1]\to\mathbb R:f^{(m-1)} \text{ is absolutely continuous and }\\
&\int_0^1(f^{(m)}(x))^2dx\leq c^2\Big\}.
\end{align*}
The \emph{periodic Sobolev space} is defined by
\begin{equation}\label{eqn_persobolev}
\tilde W(m,c)=\left\{f\in W(m,c):f^{(j)}(0)=f^{(j)}(1),\ j=0,1,\dots,m-1\right\}.
\end{equation}
The white noise model \eqref{eqn:wnm}
is asymptotically equivalent to making $n$
equally spaced observations along the sample path, $Y_i = f(i/n) +
\sigma \epsilon_i$, where $\epsilon_i \sim {\mathcal N}(0,1)$
\cite{brown1996asymptotic}.  In this formulation, the noise level in the formulation \eqref{eqn:wnm} 
scales as $\epsilon^2 = \sigma^2/n$, and the rate of convergence takes the
familiar form $n^{-\frac{2m}{2m+1}}$ where $n$ is the number of
observations.

To carry out quantized estimation we now require
an encoder
\begin{equation*}
\phi_\eps:\mathbb R^{[0,1]}\to\{1,2,\dots,2^{B_\eps}\}
\end{equation*}
which is a function applied to the sample path $X(t)$.
The decoding function then takes the form
\begin{equation*}
\psi_\eps:\{1,2,\dots,2^{B_\eps}\}\to\mathbb R^{[0,1]}
\end{equation*}
and maps the index to a function estimate.
As in the previous section, we write the composition of the encoder
and the decoder as 
$\hat f_\eps=\psi_\eps\circ\phi_\eps$, which we call the quantized estimator.
The communication or storage $C(\hat f_\eps)$ required by this quantized estimator
is no more than $B_\eps$ bits.  

To recast quantized estimation in terms of
an infinite sequence model, let
$({\varphi_j})_{j=1}^\infty$ be the trigonometric basis, and let
\begin{equation*}
\theta_j=\int_0^1 \varphi_j(t)f(t)dt,\quad j=1,2,\dots,
\end{equation*}
be the Fourier coefficients.
It is well known \cite{tsybakov2008introduction} that $f=\sum_{j=1}^\infty\theta_j\varphi_j$ belongs to $\tilde W(m,c)$ if and only if 
the Fourier coefficients $\theta$ belong to the \emph{Sobolev ellipsoid} defined as
\begin{equation}\label{eqn_sobolevellip}
\Theta(m,c)=\left\{\theta\in\ell_2: \sum_{j=1}^\infty a_j^2\theta_j^2\leq \frac{c^2}{\pi^{2m}}\right\}
\end{equation}
where 
\begin{equation*}
a_j=\begin{cases}
j^m,&\text{for even } j,\\
(j-1)^m,&\text{for odd } j.
\end{cases}
\end{equation*}
Although this is the standard definition of a Sobolev ellipsoid, 
for the rest of the paper 
we will set $a_j=j^m$, $j=1,2,\dots$ for convenience of analysis.
All of the results hold for both definitions of $a_j$.
Also note that (\ref{eqn_sobolevellip}) actually gives a more general definition,
since $m$ is no longer assumed to be an integer, as it is in (\ref{eqn_persobolev}).
Expanding with respect to the same orthonormal basis, the observed path
$X(t)$ is converted into an infinite Gaussian sequence
\begin{equation*}
Y_j=\int_0^1 \varphi_j(t) \,dX(t),\quad j=1,2,\dots,
\end{equation*}
with $Y_j\sim \mathcal N(\theta_j,\varepsilon^2)$. 
For an estimator $(\hat\theta_j)_{j=1}^\infty$ of $(Y_j)_{j=1}^\infty$, 
an estimate of $f$ is obtained by
\begin{equation*}
\hat f(x)=\sum_{j=1}^\infty \hat\theta_j\varphi_j(x)
\end{equation*}
with squared error $\|\hat f-f\|_2^2=\|\hat\theta-\theta\|_2^2$.
In terms of this standard reduction, the quantized minimax risk is thus reformulated as
\begin{equation}
R_\eps(m,c,B_\eps)=\inf_{\hat\theta_\eps,C(\hat\theta_\eps)\leq B_\eps}\sup_{\theta\in\Theta(m,c)}\E_\theta\|\theta-\hat\theta_\eps\|_2^2.\label{minimaxriskdef}
\end{equation}

To state our result, we need to define the value of the following
variational problem:
\begin{align}\label{var7}
& {\mathsf V}_{m,c,d}  \triangleq \\
\nonumber & \max_{(\sigma^2,x_0) \in
  \F(m,c,d)} \int_0^{x_0}\frac{\sigma^2(x)}{\sigma^2(x)+1}dx 
+ x_0\exp\left(\frac{1}{x_0}\int_0^{x_0}\log\frac{\sigma^4(x)}{\sigma^2(x)+1}dx-\frac{2d}{x_0}\right)
\end{align}
where the feasible set $\F(m,c,d)$ is the collection of increasing functions $\sigma^2(x)$ and
values $x_0$ satisfying
\begin{gather*}
\int_0^{x_0} x^{2m}\sigma^2(x)dx\leq c^2 \\
\frac{\sigma^4(x)}{\sigma^2(x)+1}\geq
\exp\left(\frac{1}{x_0}\int_0^{x_0}\log\frac{\sigma^4(x)}{\sigma^2(x)+1}dx-\frac{2d}{x_0}\right)
\text{ for all }x\leq x_0.
\end{gather*}
The significance and interpretation of the variational problem will
become apparent as we outline the proof of this result.

\begin{theorem}
Let $R_\eps(m,c,B_\eps)$ be defined as in
(\ref{minimaxriskdef}), for $m>0$ and $c>0$.
\begin{enumerate}[(i)]
\item 
If $B_\varepsilon\varepsilon^{\frac{2}{2m+1}}\to\infty$ as $\varepsilon\to0$, then
\begin{equation*}
\liminf_{\varepsilon\to 0} \varepsilon^{-\frac{4m}{2m+1}}R_\varepsilon(m,c,B_\varepsilon)\geq \mathsf P_{m,c}
\end{equation*}
where $\mathsf P_{m,c}$ is Pinker's constant defined in
(\ref{pinsker}).
\vskip15pt
\item If $B_\varepsilon\varepsilon^{\frac{2}{2m+1}}\to d$ for some constant $d$ as $\varepsilon\to0$, then
\begin{equation*}
\liminf_{\varepsilon\to 0}
\varepsilon^{-\frac{4m}{2m+1}}R_\varepsilon(m,c,B_\varepsilon)\geq
\mathsf P_{m,c}+\mathsf Q_{m,c,d} = {\mathsf V}_{m,c,d}
\end{equation*}
where $\mathsf V_{m,c,d}$ is the value of the variational problem \eqref{var7}.
\vskip15pt
\item If $B_\varepsilon\varepsilon^{\frac{2}{2m+1}}\to0$ and $B_\varepsilon\to\infty$ as $\varepsilon\to 0$, then
\begin{equation*}
\liminf_{\varepsilon\to 0}B_\varepsilon^{2m}R_\varepsilon(m,c,B_\varepsilon)\geq \frac{c^2m^{2m}}{\pi^{2m}}.
\end{equation*}
\end{enumerate}
\label{thm_lowerbound}
\end{theorem}

In the first regime where the number of bits $B_\varepsilon$ is much
greater than $\varepsilon^{-\frac{2}{2m+1}}$, we recover the same
convergence result as in Pinsker's theorem, in terms of both
convergence rate and leading constant. The proof of the lower bound
for this regime can directly follow the proof of Pinsker's theorem, since
the set of estimators considered in our minimax framework is a subset of
all possible estimators.

In the second regime where we have ``just enough''
bits to preserve the rate, we suffer a loss in terms of the leading
constant. In this ``Goldilocks regime,'' the optimal rate
$\varepsilon^{-\frac{4m}{2m+1}}$ is achieved but the constant in front
of the rate is Pinsker's constant $\mathsf P_{m,c}$ plus a positive
quantity $\mathsf Q_{m,c,d}$ determined by
the variational problem. 

While the solution to this variational problem does not appear to have an explicit
form, it can be computed numerically.   We discuss this term at length
in the sequel, where we explain the origin of the variational problem,
compute the constant numerically and approximate it from above and
below. The constants $\mathsf P_{m,c}$ and $\mathsf Q_{m,c,d}$ are
shown graphically in Figure \ref{fig:constant}.
Note that the parameter $d$ can be thought of as
the average number of bits per coefficient used by an optimal
quantized estimator, since
$\varepsilon^{-\frac{2}{2m+1}}$ is asymptotically the number
of coefficients needed to estimate at the classical minimax rate.
As shown in Figure~\ref{fig:constant}, the constant
for quantized estimation quickly approaches the 
Pinsker constant as $d$ increases---when $d=3$ the
two are already very close.

\begin{figure}[t]
\begin{center}
\begin{tabular}{c}
\hskip35pt
\includegraphics{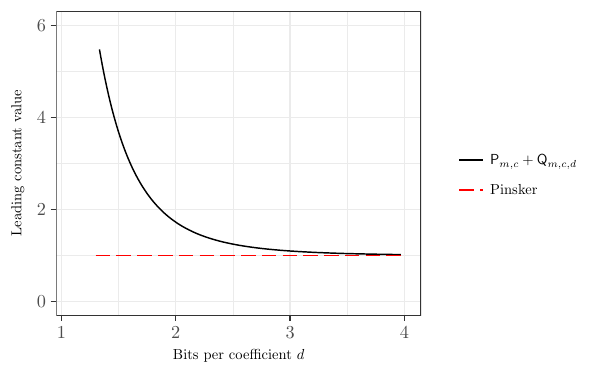} \\
\end{tabular}
\caption{The constants $\mathsf P_{m,c} + \mathsf Q_{m,c,d}$ as a
  function of quantization level $d$ in the sufficient regime,
  where $B_\varepsilon\varepsilon^{\frac{2}{2m+1}}\to d$. 
  The parameter $d$ can be thought of as
  the average number of bits per coefficient used by an optimal
  quantized estimator, because
  $\varepsilon^{-\frac{2}{2m+1}}$ is asymptotically the number
  of coefficients needed to estimate at the classical minimax rate.
  Here we take $m=2$ and $c^2/\pi^{2m}=1$.
  The curve indicates that with only $2$ bits per
  coefficient, optimal quantized minimax estimation degrades
  by less than a factor of 2 in the constant. With $3$ bits
  per coefficient, the constant is very close to the classical Pinsker
  constant.}
\label{fig:constant}
\end{center}
\end{figure}

In the third regime where the communication budget is insufficient for
the estimator to achieve the optimal rate, we obtain a sub-optimal
rate which no longer depends explicitly on the noise level
$\varepsilon$ of the model. In this regime, quantization error
dominates, and the risk decays at a rate of $B^{-\frac{1}{2m}}$ no
matter how fast $\varepsilon$ approaches zero, as long as
$B\ll\varepsilon^{-\frac{2}{2m+1}}$.  Here the analogue of Pinsker's
constant takes a very simple form.

\vskip10pt
\begin{proof}[Proof of Theorem \ref{thm_lowerbound}]
Consider a Gaussian prior
distribution on $\theta=(\theta_j)_{j=1}^\infty$ 
with $\theta_j\sim\mathcal N(0,\sigma_j^2)$ for $j=1,2,\dots,$
in terms of parameters
$\sigma^2=(\sigma_j^2)_{j=1}^\infty$ to be specified later. One
requirement for the variances is
\begin{equation*}
\sum_{j=1}^\infty a_j^2\sigma_j^2\leq \frac{c^2}{\pi^{2m}}.
\end{equation*}
We denote this prior distribution by $\pi(\theta;\sigma^2)$,
and show in Section~\ref{sec_proof} that it is asymptotically
concentrated on the ellipsoid $\Theta(m,c)$.  Under
this prior the model is
\begin{align*}
\theta_j & \sim\mathcal N(0,\sigma_j^2)\\
Y_j \given \theta_j & \sim\mathcal N(\theta_j,\varepsilon^2),\quad j=1,2,\dots\label{model}
\end{align*}
and 
the marginal distribution of $Y_j$ is thus ${\mathcal N}(0,\sigma^2_j+\eps^2)$.
Following the strategy outlined in Section~\ref{sec:general}, let
$\delta$ denote the posterior mean of $\theta$ given $Y$ under this
prior, and consider the optimization
\begin{align*}
\quad \inf \ &\ \mathbb E\|\delta-\tilde\theta\|^2\\
\text{such that}\ &\ I(\delta;\tilde\theta)\leq B_\epsilon
\end{align*}
where the infimum is over all distributions on $\tilde\theta$ such
that $\theta \to Y \to \tilde\theta$ forms a Markov chain.
Now, the posterior mean satisfies $\delta_j = \gamma_j Y_j$ where 
$\gamma_j = \sigma_j^2 / (\sigma_j^2 + \epsilon^2)$. 
Note that the Bayes risk 
under this prior is
\begin{equation*}
\E\|\theta - \delta\|_2^2 = \sum_{j=1}^\infty \frac{\sigma_j^2
  \eps^2}{\sigma_j^2 + \eps^2}.
\end{equation*}
Define
\begin{equation*}
  \mu_j^2  \triangleq \E(\delta_j - \tilde\theta_j)^2.
\end{equation*}
Then the classical rate distortion argument \cite{cover2006elements} gives that 
\begin{align*}
I(\delta; \tilde\theta) & \geq \sum_{j=1}^\infty I(\gamma_j Y_j; \tilde\theta_j) \\
&  \geq \sum_{j=1}^\infty \frac{1}{2} \log_+ \left(\frac{\gamma_j^2
  (\sigma^2_j +\eps^2)}{\mu_j^2}\right) \\
&  = \sum_{j=1}^\infty \frac{1}{2} \log_+
\left(\frac{\sigma_j^4}{\mu_j^2 (\sigma_j^2 + \eps^2)}\right)
\end{align*}
where $\log_+(x) = \max(\log x, 0)$.  
Therefore, the quantized minimax risk is lower bounded by
\begin{equation*}
R_\eps( m, c, B_\eps) = \inf_{\hat\theta_\eps,C(\hat\theta_\eps)\leq B_\eps}\sup_{\theta\in\Theta(m,c)} \mathbb
E\|\theta-\hat\theta_\eps\|^2
\geq V_\eps(B_\eps,m,c) (1+o(1))
\end{equation*}
where $V_\eps(B_\eps, m, c)$ is the value of the 
optimization
\begin{equation}
\tag{$\mathcal P_1$}\label{P1}
\begin{aligned}
\max_{\sigma^2}\min_{\mu^2}\ &\ \sum_{j=1}^\infty\mu_j^2+\sum_{j=1}^\infty\frac{\sigma_j^2\varepsilon^2}{\sigma_j^2+\varepsilon^2}\\
\text{such that}\ &\ \sum_{j=1}^\infty\frac{1}{2}\log_+\left(\frac{\sigma_j^4}{\mu_j^2(\sigma_j^2+\varepsilon^2)}\right)\leq B_\eps\\
&\ \sum_{j=1}^\infty a_j^2\sigma_j^2\leq \frac{c^2}{\pi^{2m}}
\end{aligned}
\end{equation}
and the $(1+o(1))$ deviation term is analyzed in
the supplementary material.

Observe that the quantity $V_{\varepsilon}(B_\varepsilon,m,c)$ can be upper and lower bounded by 
\begin{equation}\label{vupperlower}
\max\Bigl\{R_{\varepsilon}(m,c),Q_{\varepsilon}(m,c,B_\varepsilon)\Bigr\}
\leq V_\eps(m, c, B_\eps)\\
\leq R_{\eps}(m,c)+ Q_{\varepsilon}(m,c,B_\varepsilon)
\end{equation}
where the estimation error term $R_{\varepsilon}(m,c)$ is the value of the optimization
\begin{equation}
\tag{$\mathcal R_{1}$}\label{eqn_E1}
\begin{aligned}
\max_{\sigma^2}\ &\ \sum_{j=1}^\infty\frac{\sigma_j^2\varepsilon^2}{\sigma_j^2+\varepsilon^2}\\
\text{such that}\ &\ \sum_{j=1}^\infty a_j^2\sigma_j^2\leq \frac{c^2}{\pi^{2m}}
\end{aligned}
\end{equation}
and the quantization error term $Q_{\varepsilon}(m,c,B_\varepsilon)$ is the value of the optimization
\begin{equation}
\tag{$\mathcal Q_{1}$}\label{eqn_Q1}
\begin{aligned}
\max_{\sigma^2}\min_{\mu^2}\ &\ \sum_{j=1}^\infty\mu_j^2\\
\text{such that}\ &\ \sum_{j=1}^\infty\frac{1}{2}\log_+\left(\frac{\sigma_j^4}{\mu_j^2(\sigma_j^2+\varepsilon^2)}\right)\leq B_\eps\\
&\ \sum_{j=1}^\infty a_j^2\sigma_j^2\leq \frac{c^2}{\pi^{2m}}.
\end{aligned}
\end{equation}
The following results specify the leading order asymptotics of these quantities.
\begin{lemma} \label{lem_v1}
As $\eps\to0$,
\begin{equation*}
R_{\eps}(m,c)=\mathsf P_{m,c}\,\varepsilon^{\frac{4m}{2m+1}}(1+o(1)).
\end{equation*}
\end{lemma}
\begin{lemma} \label{lem_v2}
As $\eps\to0$,
\begin{equation}
Q_{\eps}(m,c,B_\eps)\leq \frac{c^2m^{2m}}{\pi^{2m}}B_\varepsilon^{-2m}(1+o(1)).\label{eqn_upperboundQ}
\end{equation}
Moreover, if $B_\varepsilon\varepsilon^{\frac{2}{2m+1}}\to0$ and $B_\varepsilon\to\infty$,
\begin{equation*}
Q_{\eps}(m,c,B_\eps)=\frac{c^2m^{2m}}{\pi^{2m}}B_\varepsilon^{-2m}(1+o(1)).
\end{equation*}
\end{lemma}
\noindent This yields the following closed form upper bound.
\begin{corollary}
Suppose that $B_\varepsilon\to\infty$ and $\varepsilon\to 0$. Then
\begin{equation}
\label{upperbound}
 V_\eps(m, c, B_\eps)\leq \left(\mathsf P_{m,c}\,\varepsilon^{\frac{4m}{2m+1}}+\frac{c^2m^{2m}}{\pi^{2m}}B_\varepsilon^{-2m}\right)(1+o(1)).
\end{equation}
\end{corollary}

In the insufficient regime
$B_\varepsilon\varepsilon^{\frac{2}{2m+1}}\to0$ and
$B_\varepsilon\to\infty$ as $\varepsilon\to 0$,
equation \eqref{vupperlower} and Lemma \ref{lem_v2} show that
\begin{equation*}
V_\eps(m, c, B_\eps)=\frac{c^2m^{2m}}{\pi^{2m}}B_\varepsilon^{-2m}(1+o(1)).
\end{equation*}
Similarly, in the over-sufficient regime
$B_\eps\eps^{\frac{2}{2m+1}}\to\infty$ as $\varepsilon\to 0$,
we conclude that
\begin{equation*}
V_\eps(m, c, B_\eps)=\mathsf P_{m,c}\,\varepsilon^{\frac{4m}{2m+1}}(1+o(1)).
\end{equation*}

We now turn to the sufficient regime $B_\eps \eps^{\frac{2}{2m+1}}\to d$.
We begin by making three observations about the solution to the optimization (\ref{P1}).
First, we note that the series $(\sigma^2_j)_{j=1}^\infty$ that solves
(\ref{P1}) can be assumed to be decreasing. If $(\sigma^2_j)$ were not in decreasing order, we could
rearrange it to be decreasing, and correspondingly rearrange
$(\mu^2_j)$, without violating the constraints or changing the value of the
optimization.  Second,  we note that
given $(\sigma^2_j)$, the optimal $(\mu^2_j)$ is obtained by the ``reverse
water-filling'' scheme \cite{cover2006elements}. Specifically, there exists $\eta > 0$ such that
\begin{equation*}
\mu_j^2=\begin{cases}
\eta & \text{ if } \displaystyle \frac{\sigma_j^4}{\sigma_j^2+\varepsilon^2}\geq \eta\\
\displaystyle 
\frac{\sigma_j^4}{\sigma_j^2+\varepsilon^2}& \text{ otherwise},
\end{cases}
\end{equation*}
where $\eta$ is chosen so that
\begin{equation*}
\frac{1}{2}\sum_{j=1}^\infty\log_+\left(\frac{\sigma_j^4}{\mu_j^2(\sigma_j^2+\varepsilon^2)}\right)\leq B_\eps.
\end{equation*}
Third, there exists an integer $J>0$ such that the optimal series $(\sigma^2_j)$ satisfies
\begin{equation*}
\frac{\sigma_j^4}{\sigma^2_j+\varepsilon^2}\geq \eta,\text{ for } j=1,\dots, J\quad\text{and}\quad\sigma_j^2=0, \text{ for }j>J,
\end{equation*}
where $\eta$ is the ``water-filling level'' for $(\mu^2_j)$ (see \cite{cover2006elements}).
Using these three observations, the optimization (\ref{P1}) can be reformulated
as
\begin{equation}\label{P2}
\tag{$\mathcal P_2$}
\begin{aligned}
\max_{\sigma^2,J}\ &\ J\eta+\sum_{j=1}^J\frac{\sigma_j^2\varepsilon^2}{\sigma_j^2+\varepsilon^2}\\
\text{such that}\ &\ \frac{1}{2}\sum_{j=1}^J\log_+\left(\frac{\sigma_j^4}{\eta(\sigma_j^2+\varepsilon^2)}\right)= B_\eps\\
&\ \sum_{j=1}^J a_j^2\sigma_j^2\leq \frac{c^2}{\pi^{2m}}\\
&\ (\sigma_j^2) \text{ is decreasing and
}\frac{\sigma_J^4}{\sigma_J^2+\varepsilon^2} \geq \eta.
\end{aligned}
\end{equation}

To derive the solution to (\ref{P2}), we use a continuous approximation of $\sigma^2$, writing
\begin{equation*}
\sigma^2_j=\sigma^2(jh)h^{2m+1}
\end{equation*}
where $h$ is the bandwidth to be specified and $\sigma^2(\cdot)$ is a
function defined on $(0,\infty)$. 
The constraint that $\sum_{j=1}^\infty a_j^2\sigma^2_j\leq \frac{c^2}{\pi^{2m}}$
becomes the integral constraint \cite{nussbaum1999minimax}
\begin{equation*}
\int_0^\infty x^{2m}\sigma^2(x)dx\leq \frac{c^2}{\pi^{2m}}.
\end{equation*}
We now set the bandwidth so that $h^{2m+1}=\varepsilon^2$.  This
choice of bandwidth will balance the two terms in the objective
function, and thus gives the hardest prior distribution. Applying the
above three observations under this continuous approximation, we transform
problem (\ref{P2}) to the following optimization:
\begin{equation}
\tag{$\mathcal P_3$}\label{P3}
\begin{aligned}
\quad \max_{\sigma^2,x_0}\ &\ x_0\eta+\int_0^{x_0}\frac{\sigma^2(x)}{\sigma^2(x)+1}dx\\
\text{such
  that}\ &\ \int_0^{x_0}\frac{1}{2}\log_+\left(\frac{\sigma^4(x)}{\eta(\sigma^2(x)+1)}
\right)= d\\
&\ \int_0^{x_0} x^{2m}\sigma^2(x)dx\leq \frac{c^2}{\pi^{2m}}\\
&\ \sigma^2(x)\text{ is decreasing and }\frac{\sigma^4(x)}{\sigma^2(x)+1}\geq\eta\text{ for all }x\leq x_0.
\end{aligned}
\end{equation}
Note that here we omit the convergence rate $h^{2m}=\varepsilon^{\frac{4m}{2m+1}}$ in the objective function.
The asymptotic equivalence between \eqref{P2} and \eqref{P3} can be established by a similar argument 
to Theorem 3.1 in \cite{donoho1997wald}.
Solving the first constraint for $\eta$ yields
\begin{equation}\label{P4}
\begin{aligned}
\quad \max_{\sigma^2,x_0}\ &\ \int_0^{x_0}\frac{\sigma^2(x)}{\sigma^2(x)+1}dx+x_0\exp\left(\frac{1}{x_0}\int_0^{x_0}\log\frac{\sigma^4(x)}{\sigma^2(x)+1}dx-\frac{2d}{x_0}\right)\\
\text{such that}\ &\ \int_0^{x_0} x^{2m}\sigma^2(x)dx\leq \frac{c^2}{\pi^{2m}}\\
&\ \sigma^2(x)\text{ is decreasing}\\
&\ \frac{\sigma^4(x)}{\sigma^2(x)+1}\geq
\exp\left(\frac{1}{x_0}\int_0^{x_0}\log\frac{\sigma^4(x)}{\sigma^2(x)+1}dx-\frac{2d}{x_0}\right)\\
&\  \qquad \text{ for all }x\leq x_0.
\end{aligned}
\tag{$\mathcal P_4$}
\end{equation}

The following is proved using a variational argument in the supplementary material.

\begin{lemma}\label{lem_variational}
The solution to (\ref{P4}) satisfies
\begin{equation*}
\frac{1}{(\sigma^2(x)+1)^2}+\exp\left(\frac{1}{x_0}\int_0^{x_0}\log\frac{\sigma^4(x)}{\sigma^2(x)+1}dx-\frac{2d}{x_0}\right)\frac{\sigma^2(x)+2}{\sigma^2(x)(\sigma^2(x)+1)}=\lambda x^{2m}
\end{equation*}
for some $\lambda > 0$.
\end{lemma}

Fixing $x_0$, the lemma shows that by setting
\begin{equation*}
\alpha=\exp\left(\frac{1}{x_0}\int_0^{x_0}\log\frac{\sigma^4(x)}{\sigma^2(x)+1}dx-\frac{2d}{x_0}\right)
\end{equation*}
we can express $\sigma^2(x)$ implicitly as the unique positive
root of a third-order polynomial in $y$,
\begin{equation*}
\lambda x^{2m}y^3+(2\lambda x^{2m}-\alpha)y^2+(\lambda x^{2m}-3\alpha-1)y-2\alpha.
\end{equation*}
This leads us to an explicit form of $\sigma^2(x)$ for a given value $\alpha$.
However, note that $\alpha$ still depends on $\sigma^2(x)$ and $x_0$, 
so the solution $\sigma^2(x)$ might not be compatible with $\alpha$ and $x_0$.
We can either search through a grid of values of $\alpha$ and $x_0$,
or, more efficiently, use an iterative method to find the pair of values that gives us the solution. 
We omit the details on how to calculate the values of the optimization as it is not main 
purpose of the paper. 

To summarize, in the regime $B_\eps\eps^{\frac{2}{2m+1}}\to d$ as
$\varepsilon\to 0$,
we obtain 
\begin{equation*}
V_\eps(m, c, B_\eps)=(\mathsf P_{m,c} + \mathsf Q_{m,c,d})\,\varepsilon^{\frac{4m}{2m+1}}(1+o(1)),
\end{equation*}
where we denote by $\mathsf P_{m,c} + \mathsf Q_{m,c,d}$ the 
values of the optimization (\ref{P4}).
\end{proof}


\section{Achievability} 
\label{sec:achievability}
In this section, we show that the lower bounds in Theorem
\ref{thm_lowerbound} are achievable by a quantized estimator using a
random coding scheme. The basic idea of our quantized estimation procedure is to
conduct blockwise estimation and quantization together, using a
quantized form of James-Stein estimator.

Before we present a quantized form of the James-Stein estimator, 
let us first consider a class of simple procedures. 
Suppose that $\hat \theta=\hat\theta(X)$ is an estimator of $\theta\in\Theta(m,c)$ 
without quantization. We assume that $\hat\theta\in\Theta(m,c)$, as projection
always reduces mean squared error. 
To design a $B$-bit quantized estimator, let $\check\Theta$ 
be the optimal $\delta$-covering of the parameter space $\Theta(m,c)$ such that $|\check\Theta|\leq 2^B$,
that is,
\[
\delta = \delta(B) = \inf_{\check\Theta\subset\Theta:|\check\Theta|\leq 2^B}\sup_{\theta\in\Theta}\inf_{\theta'\in\Check\Theta}\|\theta-\theta'\|.
\] 
The quantized estimator is then defined to be
\[
\check\theta=\check\theta(X) = \argmin_{\theta'\in\check\Theta}\|\hat\theta(X)-\theta'\|.
\]
Now the mean squared error satisfies
\begin{align*}
\E_\theta\|\check\theta-\theta\|^2&=\E_\theta\|\check\theta-\hat\theta+\hat\theta-\theta\|^2
\leq 2\E_\theta\|\hat\theta-\theta\|^2+2\E_\theta\|\check\theta-\hat\theta\|^2
\leq 2\sup_{\theta'}\E_{\theta'}\|\hat\theta-\theta'\|^2+2\delta(B)^2.
\end{align*}
If we pick $\hat\theta$ to be a minimax estimator for $\Theta$, 
the first term above gives the minimax risk for estimating $\theta$ in the parameter space $\Theta$. 
The second term is closely related to the metric entropy of the parameter space $\Theta(m,c)$.
In fact, for the Sobolev ellipsoid $\Theta(m,c)$, it is shown in \cite{donoho1997wald} that
$\delta(B)^2=\frac{c^2m^{2m}}{\pi^{2m}}B^{-2m}(1+o(1))$ as $B\to\infty$.
Thus, with an extra constant factor of 2, the mean squared error of this quantized estimator 
is decomposed into the minimax risk for $\Theta$ and an error term due to quantization.
In addition to the fact that this procedure does not achieve the exact lower bound
of the minimax risk for the constrained estimation problem, it is not clear how such an $\varepsilon$-net
can be generated. In what follows we will describe a quantized estimation procedure 
that we will show achieves the lower bound with the exact constants,
and that also adapts to the unknown parameters of the Sobolev space.

We begin by defining the block system to be used, which is usually referred to as the
\emph{weakly geometric system of blocks} \cite{tsybakov2008introduction}. 
Let $N_\eps=\lfloor 1/\eps^2\rfloor$ and $\rho_\varepsilon=(\log(1/\varepsilon))^{-1}$. 
Let $J_1,\dots,J_K$ be a partition of the set $\{1,\dots,N_\eps\}$ such that
\begin{align*}
\bigcup_{k=1}^K J_k=\{1,\dots,N_\eps\},\quad J_{k_1}\cap J_{k_2}=\emptyset\text{ for } k_1\neq k_2,\\
\text{and }\min\{j:j\in J_k\}>\max\{j:j\in J_{k-1}\}.
\end{align*}
Let $T_k$ be the cardinality of the $k$th block and suppose that $T_1,\dots,T_k$ satisfy
\begin{equation}\label{wgb}\begin{aligned}
T_1&=\lceil\rho_\varepsilon^{-1}\rceil=\lceil\log(1/\varepsilon)\rceil,\\
T_2&=\lfloor T_1(1+\rho_\varepsilon)\rfloor,\\
&\vdots\\
T_{K-1}&=\lfloor T_1(1+\rho_\varepsilon)^{K-2}\rfloor,\\
T_K&=N_\eps-\sum_{k=1}^{K-1}T_k.
\end{aligned}\end{equation}
Then $K \leq C\log^2(1/\eps)$ (see Lemma~\ref{lem:wgsb}).
For an infinite sequence $x\in\ell_2$, denote by $x_{(k)}$ the vector $(x_j)_{j\in J_k}\in\mathbb R^{T_k}$. We also write $j_k=\sum_{l=1}^{k-1}T_l+1$, which is the smallest index in block $J_k$.
The weakly geometric system of blocks is defined such that the size
of the blocks does not grow too quickly (the ratio between the sizes of the neighboring two 
blocks goes to 1 asymptotically), and that the number of the blocks 
is on the logarithmic scale with respect to $1/\varepsilon$ ($K\lesssim\log^2(1/\varepsilon)$).
See Lemma \ref{lem:wgsb}.

We are now ready to describe the quantized estimation scheme.  We
first give a high-level description of the scheme, and then the
precise specification. In contrast to rate distortion theory, where
the codebook and allocation of the bits are determined once the source
distribution is known, here the codebook and allocation of bits are
adaptive to the data---more bits are used for blocks having larger
signal size.  The first step in our quantization scheme is to
construct a ``base code'' of $2^{B_\eps}$ randomly generated vectors
of maximum block length $T_K$, with $\mathcal N(0,1)$ entries.  The
base code is thought of as a $2^{B_\eps} \times T_K$ random matrix
$\mathcal Z$; it is generated before observing any data, and is shared
between the sender and receiver.  After observing data $(Y_j)$, the
rows of $\mathcal Z$ are apportioned to different blocks $k=1,\ldots,
K$, with more rows being used for blocks having larger estimated
signal size.  To do so, the norm $\|Y_{(k)}\|$ of each block $k$ is
first quantized as a discrete value $\check S_k$.  A subcodebook
$\mathcal Z_k$ is then constructed by normalizing the
appropriate rows and the first $T_k$ columns of the base code,
yielding a collection of random points on the unit
sphere $\mathbb S^{T_k-1}$.  To form a quantized estimate of the
coefficients in the block, the codeword $\check Z_{(k)} \in
\mathcal Z_k$ having the smallest angle to $Y_{(k)}$ is then found.  The appropriate
indices are then transmitted to the receiver.  To decode and
reconstruct the quantized estimate, the receiver first recovers the
quantized norms $(\check S_k)$, which enables reconstruction of the
subdivision of the base code that was used by the encoder. After
extracting for each block $k$ the appropriate row of the base code,
the codeword $\check Z_{(k)}$ is reconstructed, and a
James-Stein type estimator is then calculated.

The quantized estimation scheme is detailed below.
\begin{enumerate}[{\sc Step} 1.]\setlength{\itemsep}{3pt}
\item \textit{Base code generation.}
\begin{enumerate}[{1.}1.]
\item Generate codebook $\mathcal S_k=\bigl\{\sqrt{T_k\eps^2}+i\eps^2:\ i=0,1,\dots,s_k\bigr\}$ where $s_k=\left\lceil\eps^{-2}c(j_k\pi)^{-m}\right\rceil$, for $k=1,\dots, K$.
\item Generate base code $\mathcal Z$, a $2^B\times T_K$ matrix with
  i.i.d.\ $\mathcal N(0,1)$ entries. 
\end{enumerate}
$(\mathcal S_k)$ and $\mathcal Z$ are shared between the encoder and the decoder, before seeing any data. 
\item \textit{Encoding.}
\begin{enumerate}[2.1.]
\item \textit{Encoding block radius.}
For $k=1,\dots, K$, encode \\
$\check S_k=\arg\min\left\{|s-S_k|:s\in\mathcal S_k\right\}$ where
\begin{equation*}
S_k=\begin{cases}
\sqrt{T_k\eps^2}& \text{if }\|Y_{(k)}\|<\sqrt{T_k\eps^2}\\
\sqrt{T_k\eps^2}+c(j_k\pi)^{-m} & \text{if }\|Y_{(k)}\|>\sqrt{T_k\eps^2}+c(j_k\pi)^{-m}\\
\|Y_{(k)}\|& \text{otherwise.}
\end{cases}
\end{equation*}
\item \textit{Allocation of bits.}
Let $(\tilde b_k)_{k=1}^K$ be the solution to the optimization
\begin{equation}\begin{aligned}
\min_{\bar b}\ &\ \sum_{k=1}^K\frac{(\check S_k^2-T_k\eps^2)^2}{\check S_k^2}\cdot 2^{-2 \bar b_k}\label{allocationproblem}\\
\text{such that}\ &\ \sum_{k=1}^K T_k\bar b_k\leq B,\ \bar b_k\geq 0.
\end{aligned}
\end{equation}
\item \textit{Encoding block direction.}
Form the data-dependent codebook as follows. 

Divide the rows of $\mathcal Z$ into blocks of sizes $2^{\lceil T_1\tilde b_1\rceil},\dots, 2^{\lceil T_K\tilde b_K\rceil}$. 
Based on the $k$th block of rows, construct the data-dependent codebook $\tilde {\mathcal Z}_k$ by keeping only the first $T_k$ entries and normalizing each truncated row; specifically, the $j$th row of $\tilde {\mathcal Z}_k$ is given by
\[
\tilde {\mathcal Z}_{k,j}=\frac{\mathcal Z_{i,1:T_k}}{\|\mathcal Z_{i,1:T_k}\|}\in\mathbb S_{T_k-1}
\]
where $i$ is the appropriate row of the base code $\mathcal Z$ and $\mathcal Z_{i,1:t}$ denotes the first $t$ entries of the row vector. A graphical illustration is shown below in Figure \ref{fig_codebook}. 

With this data-dependent codebook, encode 
\[
\check Z_{(k)}=\argmax\{\langle z,Y_{(k)}\rangle:z\in\tilde{\mathcal Z}_k\}
\]
for $k=1,\dots,K$.

\vskip20pt
\begin{figure}[H]
\begin{center}
\includegraphics{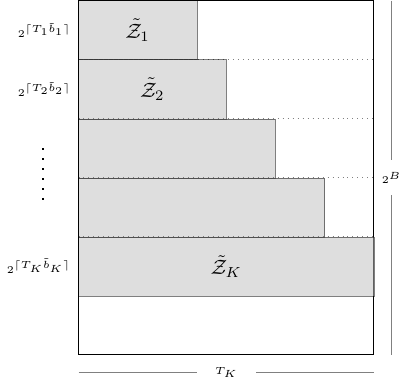}
\end{center}
\caption{An illustration of the data-dependent codebook. The big matrix represents the base code $\mathcal Z$, and the shaded areas are $(\tilde{\mathcal Z}_k)$, sub-matrices of size $T_k\times 2^{\lceil T_k\tilde b_k\rceil}$ with rows normalized.}\label{fig_codebook}
\end{figure}

\end{enumerate}

\item \textit{Transmission}.  Transmit or store $(\check S_k)_{k=1}^K$ and $(\check Z_{(k)})_{k=1}^K$ by their corresponding indices.
\item \textit{Decoding \& Estimation.} 
\begin{enumerate}[4.1.]
\item Recover $(\check S_k)$ based on the transmitted or stored indices and the common codebook $(\mathcal S_k)$.
\item Solve (\ref{allocationproblem}) and get $(\tilde b_k)$. Reconstruct $(\tilde{\mathcal Z}_k)$ using $\mathcal Z$ and $(\tilde b_k)$.
\item Recover $(\check Z_{(k)})$ based on the transmitted or stored indices and the reconstructed codebook $(\tilde{\mathcal Z}_k)$.
\item Estimate $\theta_{(k)}$ by
\begin{equation*}
\check \theta_{(k)}=\frac{\check S_k^2-T_k\eps^2}{\check S_k}\sqrt{1-2^{-2\tilde b_k}}\cdot \check Z_{(k)}.
\end{equation*}
\item Estimate the entire vector $\theta$ by concatenating the
  $\check\theta_{(k)}$ vectors and padding with zeros; thus,
\begin{equation*}
\check\theta = \left(\check\theta_{(1)},\dots,\check\theta_{(K)},0,0,\dots\right).
\end{equation*}
\end{enumerate}
\end{enumerate}

The following theorem establishes the asymptotic optimality of this
quantized estimator.

\begin{theorem}\label{thm_upperbound} Let $\check\theta$ be the quantized estimator defined above. 
\begin{enumerate}[(i)]
\item If $B\varepsilon^{\frac{2}{2m+1}}\to \infty$, then 
\begin{equation*}
\lim_{\varepsilon\to 0}\,\varepsilon^{-\frac{4m}{2m+1}}\sup_{\theta\in\Theta(m,c)}\mathbb E\|\theta-\check\theta\|^2= \mathsf P_{m,c}.
\end{equation*}
\item If $B\varepsilon^{\frac{2}{2m+1}}\to d$ for some constant $d$ as $\varepsilon\to0$, then
\begin{equation*}
\lim_{\varepsilon\to 0}\,\varepsilon^{-\frac{4m}{2m+1}}\sup_{\theta\in\Theta(m,c)}\mathbb E\|\theta-\check\theta\|^2= \mathsf P_{m,c}+\mathsf Q_{d,m,c}.
\end{equation*}
\item If $B\eps^{\frac{2}{2m+1}}\to 0$ and $B(\log(1/\eps))^{-3}\to\infty$, then 
\begin{equation*}
\lim_{\varepsilon\to 0}\,B^{2m}\sup_{\theta\in\Theta(m,c)}\mathbb E\|\theta-\check\theta\|^2=\frac{c^2m^{2m}}{\pi^{2m}}.
\end{equation*}
\end{enumerate}
The expectations are with respect to the random quantized estimation scheme $Q$ and the distribution of the data. 
\end{theorem}
We pause to make several remarks on this result before outlining the proof.
\begin{remark}
The total number of bits used by this quantized estimation scheme is 
\begin{align*}
\sum_{k=1}^K \lceil T_k \tilde b_k \rceil +\sum_{k=1}^K
\log\lceil\eps^{-2}c(j_k\pi)^{-m}\rceil & \leq \sum_{k=1}^K \lceil T_k
\tilde b_k \rceil +\sum_{k=1}^K\log\lceil\eps^{-2}c\rceil\\
& {} \leq B+ K + 2K\rho_\eps^{-1}+K\log\lceil c\rceil\\
& {} =B+O((\log(1/\eps))^{3}),
\end{align*}
where we use the fact that $K\lesssim \log^2(1/\varepsilon^2)$ (See Lemma \ref{lem:wgsb}).
Therefore, as long as $B(\log(1/\eps))^{-3}\to\infty$, the total
number of bits used is asymptotically no more than $B$, the given communication budget.
\end{remark}
\begin{remark}
The quantized estimation scheme does not make essential use of the
parameters of the Sobolev space, namely the smoothness $m$ and the
radius $c$. The only exception is that in Step 1.1 the size of the
codebook $\mathcal S_k$ depends on $m$ and $c$. However, suppose that
we know a lower bound on the smoothness $m$, say $m\geq m_0$, and an
upper bound on the radius $c$, say $c\leq c_0$. By replacing $m$ and
$c$ by $m_0$ and $c_0$ respectively, we make the codebook independent
of the parameters.  We shall assume $m_0 >1/2$, which leads 
to continuous functions.
This modification does not, however, 
significantly increase the number of bits; in fact, the total number of
bits is still $B+O(\rho_\eps^{-3})$. Thus, we can easily make this quantized
estimator minimax adaptive to the class of Sobolev ellipsoids
$\{\Theta(m,c):m\geq m_0,\; c\leq c_0\}$, as long as $B$ grows faster
than $(\log(1/\eps))^3$.  More formally, we have
\begin{corollary}
Suppose that $B_\eps$ satisfies
$B_\eps(\log(1/\eps))^{-3}\to\infty$. Let $\check\theta'$ be the
quantized estimator with the modification described above, which does
not assume knowledge of $m$ and $c$. Then for
$m\geq m_0$ and $c\leq c_0$,
\begin{equation*}
\lim_{\eps\to 0}\frac{\sup_{\theta\in\Theta(m,c)}\mathbb E\|\theta-\check\theta'\|^2}{\inf_{\hat\theta,C(\hat\theta)\leq B}\sup_{\theta\in\Theta(m,c)}\mathbb E\|\theta-\hat\theta\|^2}= 1,
\end{equation*}
where the expectation in the numerator is with respect to the data and
the randomized coding scheme, while the expectation in the denominator is only with respect to the data. 
\end{corollary}

\end{remark}
\begin{remark}
When $B$ grows at a rate comparable to or slower than
$(\log(1/\eps))^3$, the lower bound is still achievable, just no
longer by the quantized estimator we described above. The main reason
is that when $B$ does not grow faster than $\log(1/\eps)^3$, the
block size $T_1=\lceil\log(1/\eps)\rceil$ is too large.  
The blocking needs to be modified to get achievability in this case.
\end{remark}

\begin{remark}
In classical rate distortion \cite{cover2006elements,gallager1968information}, the probabilistic method applied to a
randomized coding scheme shows the existence of a code achieving 
the rate distortion bounds.  Comparing to
Theorem~\ref{thm_lowerbound},  we see that the expected risk, averaged
over the randomness in the codebook, similarly achieves the quantized
minimax lower bound.  However, note that the average over
the codebook is inside the supremum over the Sobolev space,
implying that the code achieving the bound may vary 
over the ellipsoid.  In other words, while the coding scheme
generates a codebook that is used for different
$\theta$,  it is not known whether there is one code generated by this
randomized scheme that is ``universal,'' and achieves the
risk lower bound with high probability over the ellipsoid.
The existence or non-existence of such ``universal codes'' is an 
interesting direction for further study.
\end{remark}

\begin{remark}
We have so far dealt with the periodic case, i.e., functions in the
periodic Sobolev space $\tilde W(m,c)$ defined in
(\ref{eqn_persobolev}).  For the Sobolev space $W(m,c)$, where the
functions are not necessarily periodic, the lower bound given in
Theorem \ref{thm_lowerbound} still holds, since $\tilde W(m,c)$ is a
subset of the larger class $W(m,c)$.  To extend the achievability
result to $W(m,c)$, we again need to relate $W(m,c)$ to an ellipsoid.
Nussbaum \cite{nussbaum1985spline} shows using spline
theory that the non-periodic space can actually be expressed as an
ellipsoid, where the length of the $j$th principal axis scales as $(\pi^2j)^m$
asymptotically.  Based on this link between $W(m,c)$ and the
ellipsoid, the techniques used here to show achievability apply, 
and since the principal axes scale as in the periodic case,
the convergence rates remain the same.
\end{remark}

\paragraph{Proof of Theorem \ref{thm_upperbound}} We now sketch the
proof of Theorem \ref{thm_upperbound}, deferring the full details to
Section \ref{sec_proof}.  To provide only an informal outline of the
proof, we shall write $A_1\approx A_2$ as a shorthand for $A_1=A_2(1+o(1))$,
and $A_1\lesssim A_2$ for $A_1\leq A_2(1+o(1))$, without specifying
here what these $o(1)$ terms are.

To upper bound the risk $\mathbb E\|\check\theta-\theta\|^2$, we adopt
the following sequence of approximations and inequalities. First, we discard the components whose index is greater than $N$ and show that
\begin{align*}
\hspace{-0.5in}&\mathbb E\|\check\theta-\theta\|^2\approx \mathbb E\sum_{k=1}^K\|\check\theta_{(k)}-\theta_{(k)}\|^2.\\
\intertext{Since $\check S_k$ is close enough to $S_k$, we can then safely replace $\check\theta_{(k)}$ by $\hat\theta_{(k)}=\frac{S_k^2-T_k\eps^2}{S_k}\sqrt{1-2^{-2\tilde b_{(k)}}}\cdot \check Z_{(k)}$ and obtain}
&\approx \mathbb E\sum_{k=1}^K\|\hat\theta_{(k)}-\theta_{(k)}\|^2.\\
\intertext{Writing $\lambda_k = \frac{S_k^2-T_k\eps^2}{S_k^2}$, we further decompose the risk into}
&\begin{aligned}&=\mathbb E\sum_{k=1}^K\Big(\|\hat\theta_{(k)}-\lambda_k Y_{(k)}\|^2+\|\lambda_k Y_{(k)}-\theta_{(k)}\|^2\\
&\hspace{1in}+2\langle\hat\theta_{(k)}-\lambda_k Y_{(k)},\lambda_k Y_{(k)}-\theta_{(k)} \rangle\Big).
\end{aligned}\\
\intertext{Conditioning on the data $Y$ and taking the expectation
  with respect to the random codebook yields}
&\lesssim \mathbb E\sum_{k=1}^K\left(\frac{(S_k^2-T_k\eps^2)^2}{S_k^2}2^{-2\tilde b_k}+\|\lambda_k Y_{(k)}-\theta_{(k)}\|^2\right).\\
\intertext{By two oracle inequalities upper bounding the expectations with respect to the data, and the fact that $\tilde b$ is the solution to \eqref{allocationproblem}, }
&\lesssim\min_{b\in\Pi_{\text{blk}}(B)}\sum_{k=1}^K\left(\frac{\|\theta_{(k)}\|^4}{\|\theta_{(k)}\|^2+T_k\varepsilon^2} 2^{-2 \bar b_{k}}+\frac{\|\theta_{(k)}\|^2T_k\varepsilon^2}{\|\theta_{(k)}\|^2+T_k\varepsilon^2}\right).\\
\intertext{Showing that the blockwise constant oracles are almost as good as the monotone oracle, we get for some $B'\approx B$}
&\lesssim\min_{b\in\Pi_{\text{mon}}(B'),\ \omega\in\Omega_{\text{mon}}}\sum_{j=1}^N\left(\frac{\theta_j^4}{\theta_j^2+\varepsilon^2}2^{-2 b_{j}}+(1-\omega_j)^2\theta_j^2+\omega_j^2\varepsilon^2\right),
\end{align*}
where $\Pi_{\text{blk}}(B)$, $\Pi_{\text{mon}}(B)$ are the classes of blockwise constant and monotone allocations of the bits defined in \eqref{eqn_piblk}, \eqref{eqn_pimon}, and $\Omega_{\text{mon}}$ is the class of monotone weights defined in (\ref{eqn_omegamon}).
The proof is then completed by Lemma \ref{lem_equivalence} showing that the last quantity is equal to $V_\eps(m,c,B)$.

\section{Simulations} \label{sec:experiments}
\begin{figure}[!t]
\begin{center}
\includegraphics{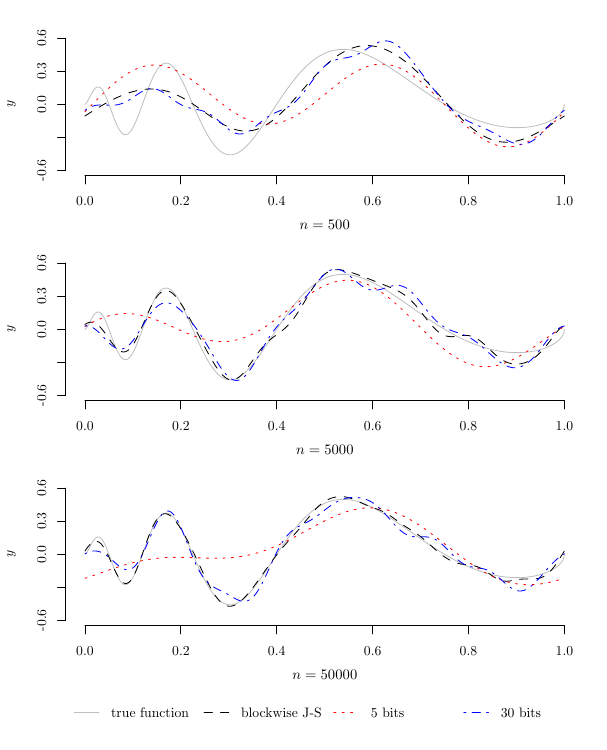} 
\caption{The damped Doppler function (solid gray) and typical realizations of the estimators 
under different noise levels ($n=500$, 5000, and 50000). 
Three estimators are used: the blockwise James-Stein estimator (dashed black),
and two quantized estimator with budgets of 5 bits (dashed red) and 30 bits (dashed blue).
}
\label{fig:sim1}
\end{center}
\end{figure}
\begin{figure}[!t]
\begin{center}
\includegraphics{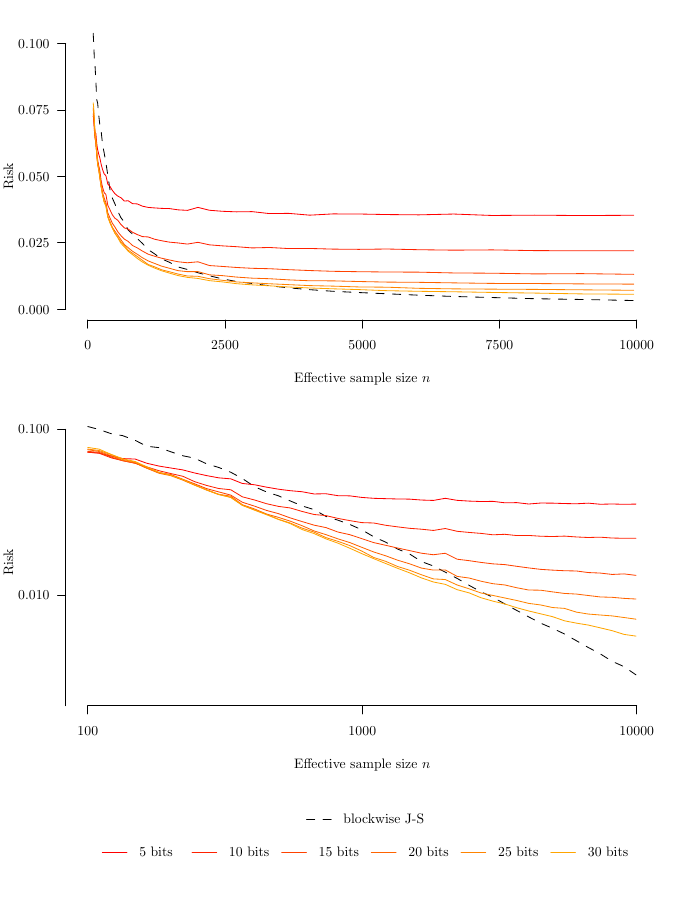} 
\caption{Risk versus effective sample size $n=1/\eps^2$
for estimating the damped Doppler function with different estimators. 
The dashed line represents the risk of the blockwise James-Stein estimator, and the solid ones
are for the quantized estimators with different budgets. 
The budgets are 5, 10, 15, 20, 25, and 30 bits, corresponding to the lines from top to bottom.
The two plots are the same curves on the original scale and the log-log scale.
}
\label{fig:sim2}
\end{center}
\end{figure}

Here we illustrate the performance of the proposed quantized estimation scheme. 
We use the function 
\[
f(x)=\sqrt{x(1-x)}\sin\left(\frac{2.1\pi}{x+0.3}\right),\quad 0\leq x\leq 1,
\]
which we shall refer to as the ``damped Doppler function,'' shown
in Figure \ref{fig:sim1} (the gray lines).  Note that the value 
$0.3$ differs from the value $0.05$ in the usual Doppler 
function used to illustrate spatial adaptation of methods such
as wavelets. Since we do not address spatial adaptivity in this
paper, we ``slow'' the oscillations of the Doppler function near zero
in our illustrations.

We use this $f$ as the underlying true mean function and generate our data according to the corresponding white noise model (\ref{eqn:wnm}),
\[
dX(t)=f(t)dt+\varepsilon dW(t), \quad 0\leq t\leq 1.
\]
We apply the blockwise James-Stein estimator,
as well as the proposed quantized estimator with different communication budgets.
We also vary the noise level $\varepsilon$ and, equivalently, the effective sample size $n=1/\varepsilon^2$.

We first show in Figure \ref{fig:sim1} some typical realizations of
these estimators on data generated under different noise levels
($n=500$, 5000, and 50000 respectively).  To keep the plots succinct,
we show only the true function, the blockwise James-Stein estimates
and quantized estimates using total bit budgets of 5 and 30 bits.  We
observe, in the first plot, that both quantized estimates deviate from
the true function, and so does the blockwise James-Stein estimates.
This is when the noise is relatively large and any quantized estimate
performs poorly, no matter how large a budget is given.  Both 5
bits and 30 bits appear to be ``sufficient/over-sufficient'' here.  In
the second plot, the blockwise James-Stein estimate is close to the
quantized estimate with a budget of 30 bits, while with a
budget of 5 bits it fails to capture the fluctuations of the true
function.  Thus, a budget of 30 bits is still ``sufficient,'' but 5
bits apparently becomes ``insufficient.''  In the third plot, the
blockwise James-Stein estimate gives a better fit than the two
quantized estimates, as both budgets become ``insufficient'' to
achieve the optimal risk.

Next, in Figure \ref{fig:sim2} we plot the risk as
a function of sample size $n$, averaging 
over 2000 simulations. Note that the bottom plot is the just the first plot on a
log-log scale.  In this set of plots, we are able to observe the phase
transition for the quantized estimators.  For relatively small values
of $n$, all quantized estimators yield a similar error rate,
with risks that are close to (or even smaller than) that
of the blockwise James-Stein estimator.  This is the
over-sufficient regime---even the smallest budget suffices to achieve
the optimal risk.  As $n$ increases, the curves start to 
separate, with estimators having smaller bit budgets leading to worse risks
compared to the blockwise James-Stein estimator, and compared
to estimators with larger budgets.  This can be seen as the sufficient regime for 
the small-budget estimators---the risks are still going down, but at a slower
rate than optimal.  The six quantized estimators all end up
in the insufficient regime---as $n$ increases, their risks begin to flatten out,
while the risk of the blockwise James-Stein estimator
continues to decrease.

\section{Related work and future directions}

Concepts related to quantized nonparametric
estimation appear in multiple communities.
As mentioned in the introduction, Donoho's 1997 Wald Lectures \cite{donoho1997wald}
(on the eve of the 50th anniversary of Shannon's 1948 paper),
drew sharp parallels between rate distortion, metric
entropy and minimax rates, focusing on the same Sobolev
function spaces we treat here.  
One view of the present
work is that we take this correspondence further by
studying how the risk continuously degrades with the level of
quantization.  We have analyzed 
the precise leading order asymptotics for quantized regression
over the Sobolev spaces, showing that these rates
and constants are realized with coding schemes 
that are adaptive to the smoothness $m$ and radius $c$ of
the ellipsoid, achieving automatically the optimal rate
for the regime corresponding to those parameters given the specified
communication budget.  Our detailed analysis is possible
due to what Nussbaum \cite{nussbaum1999minimax}
calls the ``Pinsker phenomenon,'' 
refering to the fact that linear filters attain the minimax rate in the
over-sufficient regime.  It will be interesting
to study quantized nonparametric estimation in cases where the Pinsker
phenomenon does not hold, for example over Besov bodies
and different $L_p$ spaces.

Many problems of rate distortion type are similar to quantized
regression.  The standard ``reverse water filling'' construction to
quantize a Gaussian source with varying noise levels plays a key role
in our analysis, as shown in Section~\ref{sec:lowerbound}.  In our
case the Sobolev ellipsoid is an infinite Gaussian sequence model,
requiring truncation of the sequence at the appropriate level
depending on the targeted quantization and estimation error.  In the
case of Euclidean balls, Draper and Wornell \cite{draper2004side}
study rate distortion problems motivated by communication in sensor
networks; this is closely related to the problem of quantized minimax
estimation over Euclidean balls that we analyzed in
\cite{zhu2014quantized}.  The essential difference between rate
distortion and our quantized minimax framework is that in rate
distortion the quantization is carried out for a random source, while
in quantized estimation we quantize our estimate of the deterministic
and unknown basis coefficients.  Since linear estimators are
asymptotically minimax for Sobolev spaces under squared error (the
``Pinsker phenomenon''), this naturally leads to an alternative view
of quantizing the observations, or said differently, of compressing
the data before estimation.

Statistical estimation from compressed data has appeared previously in
different communities.  In \cite{zhou2009compressed} a procedure is analyzed that
compresses data by random linear transformations in the setting of sparse linear
regression. Zhang and Berger \cite{zhang1988estimation} study estimation
problems when the data are communicated from multiple sources; Ahlswede
and Csisz\'ar \cite{ahlswede1986hypothesis} consider testing problems under
communication constraints; the use of 
side information is studied by Ahlswede and Burnashev
\cite{ahlswede1990minimax}; other formulations in terms of multiterminal
information theory are given by Han and Amari \cite{han1998statistical};
nonparametric problems are considered by Raginsky in
\cite{raginsky2007learning}.
In a distributed setting the data may be divided across different
compute nodes, with distributed estimates then aggregated or
pooled by communicating with a central node.
The general ``CEO problem'' of distributed estimation was introduced by
Berger, Zhang and Viswanathan \cite{berger1996ceo}, and has been 
recently studied in parametric settings
in \cite{zhang2013information,garg2014communication}.  
These papers take the view that the data are communicated
to the statistician at a certain rate, which may introduce
distortion, and the goal is to study the degradation
of the estimation error.  In contrast, in our setting we can view
the unquantized data as being fully available to the statistician
at the time of estimation, with communication constraints
being imposed when communicating the estimated model 
to a remote location.

Finally, our quantized minimax analysis shows achievability using
random coding schemes, which are not computationally efficient.  A
natural problem is to develop practical coding schemes that come close
to the quantized minimax lower bounds.  In our view, the most promising approach
currently is to exploit source coding schemes based on greedy sparse
regression~\cite{venkataramanan2013lossy}, applying such techniques
blockwise according to the procedure we developed in
Section~\ref{sec:achievability}.

\section*{Acknowledgements}

Research supported in part by ONR grant N00014-15-1-2379, and NSF
grants DMS-1513594 and DMS-1547396.  The authors thank Andrew Barron,
John Duchi, Maxim Raginsky, Philippe Rigollet, Harrison Zhou,
and the anonymous referees for valuable comments on this work.

\setlength{\bibsep}{8pt}
\bibliographystyle{plain}
\bibliography{bib}

\begin{thebibliography}{10}

\bibitem{ahlswede1990minimax}
Rudolf Ahlswede and MV~Burnashev.
\newblock On minimax estimation in the presence of side information about
  remote data.
\newblock {\em Ann.\ Statist.}, pages 141--171, 1990.

\bibitem{ahlswede1986hypothesis}
Rudolf Ahlswede and Imre Csisz{\'a}r.
\newblock Hypothesis testing with communication constraints.
\newblock {\em IEEE Trans.\ Inform.\ Theory}, 32(4):533--542, 1986.

\bibitem{berger1996ceo}
Toby Berger, Zhen Zhang, and Harish Viswanathan.
\newblock The {CEO} problem.
\newblock {\em IEEE Trans.\ Inform.\ Theory}, 42(3):887--902, 1996.

\bibitem{brown1996asymptotic}
Lawrence~D Brown, Mark~G Low, et~al.
\newblock Asymptotic equivalence of nonparametric regression and white noise.
\newblock {\em Ann.\ Statist.}, 24(6):2384--2398, 1996.

\bibitem{bruer2014time}
John~J Bruer, Joel~A Tropp, Volkan Cevher, and Stephen Becker.
\newblock Time--data tradeoffs by aggressive smoothing.
\newblock In {\em Advances in Neural Information Processing Systems}, pages
  1664--1672, 2014.

\bibitem{chandrasekaran2013computational}
Venkat Chandrasekaran and Michael~I Jordan.
\newblock Computational and statistical tradeoffs via convex relaxation.
\newblock {\em Proc.\ Natl.\ Acad.\ Sci.\ USA}, 110(13):E1181--E1190, 2013.

\bibitem{chattamvelli1995recurrence}
R~Chattamvelli and MC~Jones.
\newblock Recurrence relations for noncentral density, distribution functions
  and inverse moments.
\newblock {\em J.\ Stat.\ Comput.\ Simul.}, 52(3):289--299, 1995.

\bibitem{cover2006elements}
Thomas~M Cover and Joy~A Thomas.
\newblock {\em Elements of Information Theory}.
\newblock Wiley-Interscience, 2006.

\bibitem{donoho1997wald}
David~L Donoho.
\newblock Wald lecture {I}: Counting bits with {K}olmogorov and {S}hannon.
\newblock Note for the Wald Lectures, 1997.

\bibitem{draper2004side}
Stark~C Draper and Gregory~W Wornell.
\newblock Side information aware coding strategies for sensor networks.
\newblock {\em IEEE Journal on Selected Areas in Communications},
  22(6):966--976, 2004.

\bibitem{galal2011energy}
Sameh Galal and Mark Horowitz.
\newblock Energy-efficient floating-point unit design.
\newblock {\em IEEE Trans.\ Comput.}, 60(7):913--922, 2011.

\bibitem{gallager1968information}
Robert~G Gallager.
\newblock {\em Information Theory and Reliable Communication}.
\newblock John Wiley \& Sons, 1968.

\bibitem{garg2014communication}
Ankit Garg, Tengyu Ma, and Huy Nguyen.
\newblock On communication cost of distributed statistical estimation and
  dimensionality.
\newblock In {\em Advances in Neural Information Processing Systems}, pages
  2726--2734, 2014.

\bibitem{han1998statistical}
Te~Sun Han and Shun-ichi Amari.
\newblock Statistical inference under multiterminal data compression.
\newblock {\em IEEE Trans.\ Inform.\ Theory}, 44(6):2300--2324, 1998.

\bibitem{johnstone2015gaussian}
Iain~M Johnstone.
\newblock Gaussian estimation: Sequence and wavelet models.
\newblock Unpublished manuscript, 2015.

\bibitem{lucic2015tradeoffs}
Mario Lucic, Mesrob~I Ohannessian, Amin Karbasi, and Andreas Krause.
\newblock Tradeoffs for space, time, data and risk in unsupervised learning.
\newblock In {\em International Conference on Artificial Intelligence and
  Statistics}, 2015.

\bibitem{nussbaum1985spline}
Michael Nussbaum.
\newblock Spline smoothing in regression models and asymptotic efficiency in
  ${L}_2$.
\newblock {\em Ann.\ Statist.}, 13(3):984--997, 1985.

\bibitem{nussbaum1996asymptotic}
Michael Nussbaum.
\newblock Asymptotic equivalence of density estimation and gaussian white
  noise.
\newblock {\em Ann.\ of Statist.}, pages 2399--2430, 1996.

\bibitem{nussbaum1999minimax}
Michael Nussbaum.
\newblock Minimax risk: Pinsker bound.
\newblock {\em Encyclopedia of Statistical Sciences}, 3:451--460, 1999.

\bibitem{raginsky2007learning}
Maxim Raginsky.
\newblock Learning from compressed observations.
\newblock In {\em IEEE Information Theory Workshop}, pages 420--425, 2007.

\bibitem{sakrison1968geometric}
D~Sakrison.
\newblock A geometric treatment of the source encoding of a {G}aussian random
  variable.
\newblock {\em IEEE Trans.\ Inform.\ Theory}, 14(3):481--486, 1968.

\bibitem{tsybakov2008introduction}
Alexandre~B Tsybakov.
\newblock {\em Introduction to Nonparametric Estimation}.
\newblock Springer Series in Statistics, 1st edition, 2008.

\bibitem{venkataramanan2013lossy}
Ramji Venkataramanan, Tuhin Sarkar, and Sekhar Tatikonda.
\newblock Lossy compression via sparse linear regression: {C}omputationally
  efficient encoding and decoding.
\newblock In {\em IEEE International Symposium on Information Theory (ISIT)},
  pages 1182--1186. IEEE, 2013.

\bibitem{zhang2013information}
Yuchen Zhang, John Duchi, Michael~I Jordan, and Martin~J Wainwright.
\newblock Information-theoretic lower bounds for distributed statistical
  estimation with communication constraints.
\newblock In {\em Advances in Neural Information Processing Systems}, pages
  2328--2336, 2013.

\bibitem{zhang1988estimation}
Zhen Zhang and Toby Berger.
\newblock Estimation via compressed information.
\newblock {\em IEEE Trans.\ Inform.\ Theory}, 34(2):198--211, 1988.

\bibitem{zhou2009compressed}
Shuheng Zhou, John Lafferty, and Larry Wasserman.
\newblock Compressed and privacy-sensitive sparse regression.
\newblock {\em IEEE Trans.\ Inform.\ Theory}, 55(2):846--866, 2009.

\bibitem{zhu2014quantized}
Yuancheng Zhu and John Lafferty.
\newblock Quantized estimation of {G}aussian sequence models in euclidean
  balls.
\newblock In {\em Advances in Neural Information Processing Systems}, pages
  3662--3670, 2014.

\end{thebibliography}

\clearpage

\begin{appendix}

\vskip50pt
\newcommand\numberthis{\addtocounter{equation}{1}\tag{\theequation}}

\section{Proofs of Technical Results}\label{sec_proof} 
In this section, we provide proofs for Theorems \ref{thm_lowerbound} and \ref{thm_upperbound}.
\subsection{Proof of Theorem \ref{thm_lowerbound}} 
We first show
\begin{lemma} 
The quantized  minimax risk is lower bounded by $V_\varepsilon(m,c,B_\eps)$, the value of the optimization \eqref{P1}. 
\end{lemma} 
\begin{proof}
As will be clear to the reader, $V_\varepsilon(m,c,B_\eps)$ is achieved by some $\sigma^2$ that is non-increasing and finitely supported. Let $\sigma^2$ be such that
\begin{equation*}
\sigma_1^2\geq\dots\geq\sigma_n^2>0=\sigma_{n+1}=\dots,\ \sum_{j=1}^na_j^2\sigma_j^2=\frac{c^2}{\pi^{2m}},
\end{equation*}
and let 
\begin{equation*}
\Theta_n(m,c)=\{\theta\in\ell_2:\sum_{j=1}^na_j^2\theta_j^2\leq\frac{c^2}{\pi^{2m}},\ \theta_j=0\text{ for }j\geq n+1\}\subset\Theta(m,c).
\end{equation*}
We build on this sequence of $\sigma^2$ a prior distribution of $\theta$.
In particular, for $\tau\in(0,1)$, write $s_j^2=(1-\tau)\sigma_j^2$ and let $\pi_\tau(\theta;\sigma^2)$ be a the prior distribution on $\theta$ such that
\begin{align*}
\theta_j\sim\mathcal N(0,s_j^2),&\quad j=1,\dots,n,\\
\mathbb P(\theta_j=0)=1,&\quad j\geq n+1.
\end{align*}
We observe that
\begin{align*}
R_\eps(m, c, B_\eps) &\geq \inf_{\hat\theta,C(\hat\theta)\leq B_\eps}\sup_{\theta\in\Theta_n(m,c)} \mathbb
E\|\theta-\hat\theta\|^2\\
&\geq \inf_{\hat\theta,C(\hat\theta)\leq B_\eps}\int_{\Theta_n(m,c)} \mathbb E\|\theta-\hat\theta\|^2d\pi_\tau(\theta;\sigma^2)\\
&\geq I_\tau-r_\tau
\end{align*}
where $I_\tau$ is the integrated risk of the optimal quantized estimator
\begin{equation*}
I_\tau= \inf_{{\hat\theta,C(\hat\theta)\leq B_\eps}}\int_{\mathbb R^n\otimes\{0\}^\infty} \mathbb E\|\theta-\hat\theta\|^2d\pi_\tau(\theta;\sigma^2)
\end{equation*}
and $r_\tau$ is the residual
\begin{equation*}
r_\tau = \sup_{\hat\theta\in\Theta(m,c)}\int_{\overline{\Theta(m,c)}} \mathbb E\|\theta-\hat\theta\|^2d\pi_\tau(\theta;\sigma^2)
\end{equation*}
where $\overline{\Theta(m,c)}=(\mathbb R^n\otimes\{0\}^\infty)\backslash\Theta_n(m,c)$.
As shown in Section \ref{sec:lowerbound}, $\lim_{\tau\to0}I_\tau$ is lower bounded by the value of the optimization
\begin{equation*}
\begin{aligned}
\min_{\mu^2}\ &\ \sum_{j=1}^\infty\mu_j^2+\sum_{j=1}^\infty\frac{\sigma_j^2\varepsilon^2}{\sigma_j^2+\varepsilon^2}\\
\text{such that}\ &\ \sum_{j=1}^\infty\frac{1}{2}\log_+\left(\frac{\sigma_j^4}{\mu_j^2(\sigma_j^2+\varepsilon^2)}\right)\leq B_\eps.
\end{aligned}
\end{equation*}
It then suffices to show that $r_\tau=o(I_\tau)$ as $\varepsilon\to 0$ for $\tau\in(0,1)$. Let $d_n=\sup_{\theta\in\Theta_n(m,c)}\|\theta\|$, which is bounded since for any $\theta\in\Theta_n(m,c)$
\begin{equation*}
\|\theta\|=\sqrt{\sum_{j}\theta_j^2} = \sqrt{\frac{1}{a_1^2}\sum_{j}a_1^2\theta_j^2}\leq\sqrt{\frac{1}{a_1^2}\sum_{j}a_j^2\theta_j^2}
\leq\sqrt{\frac{1}{a_1^2}\frac{c^2}{\pi^{2m}}} = \frac{c}{a_1\pi^m}.
\end{equation*}
We have
\begin{align*}
r_\tau&=\sup_{\hat\theta\in\Theta(m,c)}\int_{\overline{\Theta_n(m,c)}} \mathbb E\|\theta-\hat\theta\|^2d\pi_\tau(\theta;\sigma^2)\\
&\leq 2\int_{\overline{\Theta_n(m,c)}}(d_n^2+ \mathbb E\|\theta\|^2)d\pi_\tau(\theta;\sigma^2)\\
&\leq 2\left(d_n^2\, \mathbb P\left(\theta\notin{\Theta_n(m,c)}\right)+\left(\mathbb P(\theta\notin\Theta_n(m,c))\mathbb E\|\theta\|^4\right)^{1/2}\right)
\end{align*}
where we use the Cauchy-Schwarz inequality. Noticing that
\begin{align*}
\mathbb E\|\theta\|^4&=\mathbb E\bigg(\Big(\sum_{j=1}^n\theta_j^2\Big)^2\bigg)\\
&=\sum_{j_1\neq j_2}\mathbb E(\theta_{j_1}^2)\mathbb E(\theta_{j_2}^2)+\sum_{j=1}^n\mathbb E(\theta_j^4)\\
&\leq \sum_{j_1\neq j_2}s_{j_1}^2s_{j_2}^2+3\sum_{j=1}^ns_j^4\\
&\leq 3\Big(\sum_{j=1}^ns_j^2\Big)^2\leq 3d_n^4,
\end{align*}
we obtain
\begin{align*}
r_\tau&\leq 2d_n^2\left(\mathbb P\left(\theta\notin{\Theta_n(m,c)}\right)+\sqrt{3\mathbb P\left(\theta\notin{\Theta_n(m,c)}\right)}\right)\\
&\leq 6d_n^2\sqrt{\mathbb P\left(\theta\notin{\Theta_n(m,c)}\right)}.
\end{align*}
Thus, we only need to show that $\sqrt{\mathbb P\left(\theta\notin{\Theta_n(m,c)}\right)}=o(I_\tau)$. In fact,
\begin{align*}
\mathbb P(\theta\notin&\Theta_n (m,c))\nonumber\\
&=\mathbb P\left(\sum_{j=1}^na_j^2\theta_j^2>\frac{c^2}{\pi^{2m}}\right)\\
&=\mathbb P\left(\sum_{j=1}^na_j^2(\theta_j^2-\mathbb E(\theta_j^2))>\frac{c^2}{\pi^{2m}}-(1-\tau)\sum_{j=1}^na_j^2\sigma_j^2\right)\\
&=\mathbb P\left(\sum_{j=1}^na_j^2(\theta_j^2-\mathbb E(\theta_j^2))>\frac{\tau c^2}{\pi^{2m}}\right)\\
&=\mathbb P\left(\sum_{j=1}^n a_j^2s_j^2(Z_j^2-1)>\frac{\tau}{1-\tau}\sum_{j=1}^na_j^2s_j^2\right)
\end{align*}
where $Z_j\sim\mathcal N(0,1)$.
By Lemma \ref{chisqineq}, we get
\begin{equation*}
\mathbb P(\theta\notin\Theta_n (m,c))\leq \exp\left(-\frac{\tau^2}{8(1-\tau)^2}\frac{\sum_{j=1}^na_j^2s_j^2}{\max_{1\leq j\leq n}a_j^2s_j^2}\right)
=\exp\left(-\frac{\tau^2}{8(1-\tau)^2}\frac{\sum_{j=1}^na_j^2\sigma_j^2}{\max_{1\leq j\leq n}a_j^2\sigma_j^2}\right)
\end{equation*}
Next we will show that for the $\sigma^2$ that achieves $V_\varepsilon(m,c,B_\varepsilon)$, we have $\sqrt{\mathbb P(\theta\notin\Theta_n (m,c))}=o(I_\tau)$.
For the sufficient regime where $B_\varepsilon\varepsilon^{\frac{2}{2m+1}}\to\infty$ as $\varepsilon\to0$, it is shown in \cite{tsybakov2008introduction} that $\max_{1\leq j\leq n}a_j^2\sigma_j^2 = O(\eps^{\frac{2}{2m+1}})$ and $I_\tau = O(\eps^{\frac{4m}{2m+1}})$, and hence that $\sqrt{\mathbb P(\theta\notin\Theta_n (m,c))}=o(I_\tau)$.
For the insufficient regime where $B_\varepsilon\varepsilon^{\frac{2}{2m+1}}\to0$ but still $B_\varepsilon\to\infty$ as $\varepsilon\to 0$,
an achieving sequence $\sigma^2$ is given later by \eqref{eqn:insufficient_achievingsigma} and \eqref{eqn:asymp_Jeps}. 
We obtain that
$\max_{1\leq j\leq n}a_j^2\sigma_j^2 = O(B_\eps^{-1})$ and $I_\tau = O(B_\eps^{-2m})$, and therefore $\sqrt{\mathbb P(\theta\notin\Theta_n (m,c))}= o(I_\tau)$.
The sufficient regime where $B_\varepsilon\varepsilon^{\frac{2}{2m+1}}\to d$ for some constant $d$ is a bit more complicated as we don't have an explicit formula for the optimal sequence $\sigma^2$. However, by Lemma \ref{lem_variational}, for the continuous approximation $\sigma^2(x)$ such that $\sigma^2_j = \sigma^2(jh)h^{2m+1}$, we have
\begin{align*}
\lambda x^{2m}\sigma^2(x) &= \frac{\sigma^2(x)}{(\sigma^2(x)+1)^2}+\alpha\cdot\frac{\sigma^2(x)+2}{\sigma^2(x)+1}
\leq \frac{1}{4}+2\alpha
\end{align*}
where $\alpha=\exp\left(\frac{1}{x_0}\int_0^{x_0}\log\frac{\sigma^4(x)}{\sigma^2(x)+1}dx-\frac{2d}{x_0}\right)$
and $\lambda$ are both constants.
Therefore,
\begin{equation*}
\max_{1\leq j\leq n}a_j^2\sigma_j^2\approx j^{2m}\sigma^2(jh)h^{2m+1}\leq \frac{1}{\lambda}(\frac{1}{4}+2\alpha)\cdot h.
\end{equation*}
Note that $\sum_{j=1}^na_j^2\sigma^2_j = O(h^{2m})$ and that $h=\eps^{\frac{2}{2m+1}}$.
We obtain that for this case $I_\tau = O(\eps^{\frac{4m}{2m+1}})$ and $\sqrt{\mathbb P(\theta\notin\Theta_n (m,c))} = o(I_\tau)$.
Thus, for each of the three regimes, we have $r_\tau=o(I_\tau)$.
\end{proof}

\begin{lemma}[Lemma 3.5 in \cite{tsybakov2008introduction}]\label{chisqineq}
Suppose that $X_1,\dots,X_n$ are i.i.d.\ $\mathcal N(0,1)$. For $t\in(0,1)$ and $\omega_j>0$, $j=1,\dots,n$, we have
\begin{equation*}
\mathbb P\left(\sum_{j=1}^n\omega_j(X_j^2-1)>t\sum_{j=1}^nX_j\right)\leq\exp\left(-\frac{t^2\sum_{j=1}^n\omega_j}{8\max_{1\leq j\leq n}\omega_j}\right).
\end{equation*}
\end{lemma}

\begin{proof}[Proof of Lemma \ref{lem_v1}]
This is in fact Pinsker's theorem, which gives the exact asymptotic minimax risk of estimation of normal means in the Sobolev ellipsoid. The proof can be found in \cite{nussbaum1999minimax} and \cite{tsybakov2008introduction}.
\end{proof}

\begin{proof}[Proof of Lemma \ref{lem_v2}]
As argued in Section \ref{sec:lowerbound} for the lower bound in the sufficient regime, optimization problem \eqref{eqn_Q1} can be reformulated as
\begin{equation}
\tag{$\mathcal Q_2$}\label{eqn_Q2}
\begin{aligned}
\max_{\sigma^2,J}\ &\ J\eta\\
\text{such that}\ &\ \frac{1}{2}\sum_{j=1}^J\log_+\left(\frac{\sigma_j^4}{\eta(\sigma_j^2+\varepsilon^2)}\right)\leq B_\eps\\
&\ \sum_{j=1}^J a_j^2\sigma_j^2\leq \frac{c^2}{\pi^{2m}}\\
&\ (\sigma_j^2) \text{ is decreasing and
}\frac{\sigma_J^4}{\sigma_J^2+\varepsilon^2} \geq \eta.
\end{aligned}
\end{equation}
Now suppose that we have a series ($\sigma_j^2$) which satisfies the last constraint and is supported on $\{1,\dots,J\}$. 
By the first constraint, we have that
\begin{align*}
J\eta&=J\exp\left(-\frac{2B_\eps}{J}\right)\left(\prod_{j=1}^J\frac{\sigma_j^4}{\sigma_j^2+\varepsilon^2}\right)^{\frac{1}{J}}\\
&\leq J\exp\left(-\frac{2B_\eps}{J}\right)\left(\prod_{j=1}^J\sigma_j^2\right)^{\frac{1}{J}}\\
&=J\exp\left(-\frac{2B_\eps}{J}\right)\left(\prod_{j=1}^Ja_j^2\sigma_j^2\right)^{\frac{1}{J}}\left(\prod_{j=1}^Ja_j^{-2}\right)^{\frac{1}{J}}\\
&\leq \exp\left(-\frac{2B_\eps}{J}\right)\left(\sum_{j=1}^Ja_j^2\sigma_j^2\right)\left(\prod_{j=1}^Ja_j^{-2}\right)^{\frac{1}{J}}\\ 
&\leq \frac{c^2}{\pi^{2m}}\exp\left(-\frac{2B_\eps}{J}\right)\left(\prod_{j=1}^Ja_j^{-2}\right)^{\frac{1}{J}}\\ 
&=\frac{c^2}{\pi^{2m}}\left(\exp\left(\frac{B_\eps}{m}\right)J!\right)^{-\frac{2m}{J}}.\numberthis\label{eqn_upperboundsQ}
\end{align*}
This provides a series of upper bounds for 
$Q_\eps(m,c,B_\eps)$ parameterized by $J$.
To minimize \eqref{eqn_upperboundsQ} over $J$, we look at the ratio of the neighboring terms with $J$ and $J+1$, and compare it to 1. 
We obtain that the optimal $J$ satisfies
\begin{equation}
\frac{J^J}{J!}< \exp\left(\frac{B_\eps}{m}\right)\leq \frac{(J+1)^{J+1}}{(J+1)!}.\label{eqn_Jeps}
\end{equation}
Denote this optimal $J$ by $J_\eps$. By Stirling's approximation, we have
\begin{equation}\label{eqn:asymp_Jeps}
\lim_{\eps\to0}\frac{B_\eps/m}{J_\eps}=1,
\end{equation}
and plugging this asymptote into \eqref{eqn_upperboundsQ}, we get as $\eps\to0$
\begin{equation*}
\frac{c^2}{\pi^{2m}}\left(\exp\left(\frac{B_\eps}{m}\right)J_\eps!\right)^{-\frac{2m}{J_\eps}}\sim
\frac{c^2}{\pi^{2m}}J_\eps^{-2m}\sim\frac{c^2m^{2m}}{\pi^{2m}}B_\eps^{-2m}.
\end{equation*}
This gives the desired upper bound \eqref{eqn_upperboundQ}.

Next we show that the upper bound \eqref{eqn_upperboundQ} is asymptotically achievable when $B_\eps\eps^{\frac{2}{2m+1}}\to0$ and $B_\eps\to\infty$. It suffices to find a feasible solution that attains \eqref{eqn_upperboundQ}. Let 
\begin{equation}\label{eqn:insufficient_achievingsigma}
\tilde\sigma^2_j = \frac{c^2/\pi^{2m}}{J_\eps a_j^2},\ j=1,\dots, J_\eps.
\end{equation}
Note that the entire sequence of $(\tilde\sigma^2_j)_{j=1}^{J_\eps}$ does not qualify for a feasible solution,
since the first constraint in \eqref{eqn_Q2} won't be satisfied for any $\eta\leq\frac{\tilde\sigma_{J_\eps}^4}{\tilde\sigma_{J_\eps}^2+\eps^2}$. We keep only the first $J_\eps'$ terms of $(\tilde\sigma^2_j)$, where $J_\eps'$ is the largest $j$ such that 
\begin{equation}
\frac{\tilde\sigma^4_j}{\tilde\sigma^2_{j}+\eps^2}\geq \tilde\sigma^2_{J_\eps}.\label{eqn_Jepsprime}
\end{equation}
Thus,
\begin{equation*}
\sum_{j=1}^{J_\eps'}\frac{1}{2}\log_+\left(\frac{\frac{\tilde\sigma^4_j}{\tilde\sigma_j^2+\eps^2}}{\tilde\sigma_{J_\eps}^2}\right)\leq \sum_{j=1}^{J_\eps'}\frac{1}{2}\log_+\left(\frac{\tilde\sigma_j^2}{\tilde\sigma_{J_\eps}^2}\right)\leq \sum_{j=1}^{J_\eps}\frac{1}{2}\log_+\left(\frac{\tilde\sigma_j^2}{\tilde\sigma_{J_\eps}^2}\right)\leq B_\eps,
\end{equation*}
where the last inequality is due to \eqref{eqn_Jeps}. This tells us that setting $\eta=\tilde\sigma_{J_\eps}^2$ leads to a feasible solution to \eqref{eqn_Q2}. As a result, 
\begin{equation}
Q_\eps(m,c,B_\eps)\geq J'_\eps\tilde\sigma^2_{J_\eps}.\label{eqn_achieveQ}
\end{equation}
If we can show that $J_\eps'\sim J_\eps$, then 
\begin{equation}
J'_\eps\tilde\sigma^2_{J_\eps}\sim J_\eps\tilde\sigma^2_{J_\eps}\sim\frac{c^2m^{2m}}{\pi^{2m}}B_\eps^{-2m}.
\end{equation}
To show that $J_\eps'\sim J_\eps$, it suffices to show that $a_{J_\eps'}\sim a_{J_\eps}$. Plugging the formula of $\tilde\sigma^2_j$ into (\ref{eqn_Jepsprime}) and solving for $a^2_{J_\eps'}$, we get 
\begin{align*}
a^2_{J_\eps'}&\sim\frac{\displaystyle -\frac{c^2}{\pi^{2m}J_\eps}+\sqrt{(\frac{c^2}{\pi^{2m}J_\eps})^2+4\frac{c^2}{\pi^{2m}J_\eps}\eps^2a_{J_\eps}^2}}{2\eps^2}\\
&\sim\frac{\displaystyle -\frac{c^2}{\pi^{2m}J_\eps}+\frac{c^2}{\pi^{2m}J_\eps}+\frac{1}{2}\frac{\pi^{2m}J_\eps}{c^2}4\frac{c^2}{\pi^{2m}J_\eps}\eps^2a_{J_\eps}^2}{2\eps^2} = a_{J_\eps}^2
\end{align*}
where the equivalence is due to the assumption $B_\eps\eps^{\frac{2}{2m+1}}\to0$ and a Taylor's expansion of the function $\sqrt{x}$.
\end{proof}

\begin{proof}[Proof of Lemma \ref{lem_variational}]
Suppose that $\sigma^2(x)$ with $x_0$ solves \eqref{P4}. 
Consider function $\sigma^2(x)+\xi v(x)$ such that it is still feasible for \eqref{P4},
and thus we have
\begin{equation*}
\int_0^{x_0}x^{2m}v(x)dx\leq 0.
\end{equation*}
Now plugging $\sigma^2(x)+\xi v(x)$ for $\sigma^2(x)$ in the objective function of \eqref{P4}, taking derivative with respect to $\xi$, and letting $\xi\to0$, we must have
\begin{equation*}
\int_0^{x_0}\frac{v(x)}{(\sigma^2(x)+1)^2}dx+x_0\exp\left(\frac{1}{x_0}\int_0^{x_0}\log\frac{\sigma^4(x)}{\sigma^2(x)+1}dx-\frac{2d}{x_0}\right)\frac{1}{x_0}\int_0^{x_0}\frac{2v(x)}{\sigma^2(x)}-\frac{v(x)}{\sigma^2(x)+1}dx\leq 0,
\end{equation*}
which, after some calculation and rearrangement of terms, yields
\begin{equation*}
\int_0^{x_0}v(x)\left(\frac{1}{(\sigma^2(x)+1)^2}+\exp\left(\frac{1}{x_0}\int_0^{x_0}\log\frac{\sigma^4(x)}{\sigma^2(x)+1}dx-\frac{2d}{x_0}\right)\frac{\sigma^2(x)+2}{\sigma^2(x)(\sigma^2(x)+1)}\right)dx\leq0.
\end{equation*}
Thus, by the lemma that follows, we obtain that for some $\lambda$
\begin{equation*}
\frac{1}{(\sigma^2(x)+1)^2}+\exp\left(\frac{1}{x_0}\int_0^{x_0}\log\frac{\sigma^4(y)}{\sigma^2(y)+1}dy-\frac{2d}{x_0}\right)\frac{\sigma^2(x)+2}{\sigma^2(x)(\sigma^2(x)+1)}=\lambda x^{2m}.
\end{equation*}
\end{proof}

\begin{lemma}\label{lem_variational2}
Suppose that $f(x)$ and $g(x)$ are two non-zero functions on $(0,x_0)$ such that for any $v(x)$ satisfying
$\int_0^{x_0}f(x)v(x)dx \leq 0$,
it holds that
$\int_0^{x_0}g(x)v(x)dx \leq 0$.
Then there exists a constant $\lambda$ such that $f(x)=\lambda g(x)$.
\end{lemma}
\begin{proof}
First we show that for any $v(x)$ such that $\int_0^{x_0}f(x)v(x)dx=0$ we must have $\int_0^{x_0}g(x)v(x)dx=0$. 
Otherwise, suppose that $v_0(x)$ is such that $\int_0^{x_0}f(x)v_0(x)dx=0$ and $\int_0^{x_0}g(x)v_0(x)dx<0$.
Then take another $v(x)$ with $\int_0^{x_0}f(x)v(x)dx\leq 0$ and consider $v_\gamma(x) = v(x)-\gamma v_0(x)$.
We have $\int_0^{x_0}f(x)v_\gamma(x)dx\leq 0$ and 
$\int_0^{x_0}g(x)v_\gamma(x) = \int_0^{x_0}v(x)g(x)dx-\gamma \int_0^{x_0}g(x)v_0(x)dx>0$ for large enough $\gamma$,
which results in contradiction. 

Let $\lambda = \int_0^{x_0} f(x)^2dx / \int_0^{x_0}f(x)g(x)dx$ as the denominator cannot be zero. 
In fact, if $\int_0^{x_0}f(x)g(x)dx=0$, it would imply that $\int_0^{x_0}g(x)^2dx=0$ and hence $g(x)\equiv 0$.
Now consider the function $f(x)-\lambda g(x)$. Notice that we have
$\int_0^{x_0}f(x)(f(x)-\lambda g(x))dx=0$ by the definition of $\lambda$.
It follows that $\int_0^{x_0}g(x)(f(x)-\lambda g(x))dx=0$, and therefore,
$\int_0^{x_0}(f(x)-\lambda g(x))^2dx=0$, which concludes the proof. 
\end{proof}

\subsection{Proof of Theorem \ref{thm_upperbound}}

Now we give the details of the proof of
Theorem \ref{thm_upperbound}. For the purpose of our analysis, we
define two allocations of bits, the monotone allocation and the blockwise constant allocation,
\begin{align}
\Pi_{\text{blk}}(B)&=\left\{(b_j)_{j=1}^\infty:\ \sum_{j=1}^\infty b_j\leq B,\ b_j= \bar b_{k}\text{ for }j\in J_k,\ 0\leq b_j\leq b_{\max}\right\},\label{eqn_piblk}\\
\Pi_{\text{mon}}(B)&=\left\{(b_j)_{j=1}^\infty:\ \sum_{j=1}^\infty b_j\leq B,\ b_{j-1}\geq b_j,\ 0\leq b_j\leq b_{\max}\right\},\label{eqn_pimon}
\end{align}
where $b_{\max}=2\log(1/\eps)$.  We also define two classes of weights, the monotonic weights and the blockwise constant weights,
\begin{align}
\Omega_{\text{blk}}&=\left\{(\omega_j)_{j=1}^\infty:\ \omega_j= \bar\omega_{k}\text{ for }j\in J_k,\ 0\leq \omega_j\leq 1\right\},\label{eqn_omegablk}\\
\Omega_{\text{mon}}&=\left\{(\omega_j)_{j=1}^\infty:\ \omega_{j-1}\geq \omega_j,\ 0\leq \omega_j\leq 1\right\}.\label{eqn_omegamon}
\end{align}
We will also need the following results from \cite{tsybakov2008introduction} regarding the weakly geometric system of blocks.
\begin{lemma}\label{lem:wgsb}
Let $\{J_k\}$ be a weakly geometric block system defined by \eqref{wgb}. Then there exists $0<\eps_0<1$ and $C>0$ such that for any $\eps\in(0,\eps_0)$, 
\begin{align*}
K\leq C\log^2(1/\eps),\\
\max_{1\leq k\leq K-1}\frac{T_{k+1}}{T_k}\leq 1+3\rho_\eps.
\end{align*}
\end{lemma}
We divide the proof into four steps.

\subsection*{Step 1. Truncation and replacement}
The loss of the quantized estimator $\check\theta$ can be decomposed into
\begin{equation*}
\|\check\theta-\theta\|^2=\sum_{k=1}^{K}\|\check\theta_{(k)}-\theta_{(k)}\|^2+\sum_{j=N+1}^\infty\theta_j^2,
\end{equation*}
where the remainder term satisfies
\begin{equation*}
\sum_{j=N+1}^\infty \theta_j^2\leq N^{-2m}\sum_{j=N+1}^\infty a_j^2\theta_j^2=O(N^{-2m}).
\end{equation*}
If we assume that $m>1/2$, which corresponds to classes of continuous functions, 
the remainder term is then $o(\eps^2)$.
If $m\leq 1/2$, the remainder term is on the order of $O(\eps^{4m})$, 
which is still negligible compared to the order of the lower bound $\eps^{\frac{4m}{2m+1}}$. 
To ease the notation, we will assume that $m>1/2$, 
and write the remainder term as $o(\eps^2)$, 
but need to bear in mind that the proof works for all $m>0$.
We can thus discard the remainder term in our analysis. 
Recall that the quantized estimate for each block is given by
\begin{equation*}
\check \theta_{(k)}=\frac{\check S_k^2-T_k\eps^2}{\check S_k}\sqrt{1-2^{-2\tilde b_k}} \check Z_{(k)},
\end{equation*}
and consider the following estimate with $\check S_k$ replaced by $S_k$
\begin{equation*}
\hat\theta_{(k)}=\frac{S_k^2-T_k\eps^2}{S_k}\sqrt{1-2^{-2\tilde b_k}}\check Z_{(k)}.\label{eqn_thetahat}
\end{equation*} 
Notice that
\begin{align*}
\|\hat\theta_{(k)}-\check\theta_{(k)}\|&=\left|\frac{\check S_k^2-T_k\eps^2}{\check S_k}-\frac{S_k^2-T_k\eps^2}{S_k}\right|\sqrt{1-2^{-2\tilde b_k}}\|\check Z_{(k)}\|\\
&\leq \left|\frac{\check S_kS_k+T_k\eps^2}{\check S_kS_k}\right|\left|\check S_k-S_k\right|\\
&\leq 2\eps^2
\end{align*}
where the last inequality is because $\check S_kS_k\geq T_k\eps^2$ and $\left|\check S_k-S_k\right|\leq \eps^2$. Thus we can safely replace $\check\theta_{(k)}$ by $\hat\theta_{(k)}$ because
\begin{align*}
\|\check\theta_{(k)}&-\theta_{(k)}\|^2\nonumber\\
&=\|\check\theta_{(k)}-\hat\theta_{(k)}+\hat\theta_{(k)}-\theta_{(k)}\|^2\\
&\leq \|\check\theta_{(k)}-\hat\theta_{(k)}\|^2+\|\hat\theta_{(k)}-\theta_{(k)}\|^2+2\|\check\theta_{(k)}-\hat\theta_{(k)}\|\|\hat\theta_{(k)}-\theta_{(k)}\|\\
&= \|\hat\theta_{(k)}-\theta_{(k)}\|^2+O(\eps^2).
\end{align*}
Therefore, we have
\begin{equation*}
\mathbb E\|\check\theta-\theta\|^2=\mathbb E\sum_{k=1}^K\|\hat\theta_{(k)}-\theta_{(k)}\|^2+O(K\eps^2).
\end{equation*}

\subsection*{Step 2. Expectation over codebooks}
Now conditioning on the data $Y$, we work under the probability measure introduced by the random codebook. Write
\begin{equation*}
\lambda_k = \frac{S_k^2-T_k\eps^2}{S_k^2} \text{ and  } Z_{(k)}=\frac{Y_{(k)}}{\|Y_{(k)}\|}.
\end{equation*}
We decompose and examine the following term
\begin{align*}
A_k&=\|\hat\theta_{(k)}-\theta_{(k)}\|^2\\
&=\|\hat\theta_{(k)}-\lambda_k S_kZ_{(k)}+\lambda_k S_kZ_{(k)}-\theta_{(k)}\|^2\\
&=\underbrace{\|\hat\theta_{(k)}-\lambda_k S_kZ_{(k)}\|^2}_{A_{k,1}}+\underbrace{\|\lambda_k S_kZ_{(k)}-\theta_{(k)}\|^2}_{A_{k,2}}\nonumber\\
&\hspace{1in}+\underbrace{2\langle\hat\theta_{(k)}-\lambda_k S_kZ_{(k)},\lambda_k S_kZ_{(k)}-\theta_{(k)}\rangle}_{A_{k,3}}.
\end{align*}
To bound the expectation of the first term $A_{k,1}$, we need the following lemma, which bounds the probability of the distortion of a codeword exceeding the desired value. 
\begin{lemma}\label{lem_ratedistortion}
Suppose that $Z_1,\dots,Z_n$ are independent and each follows the uniform distribution on the $t$-dimensional unit sphere $\mathbb S^{t-1}$. Let $y\in \mathbb S^{t-1}$ be a fixed vector, and 
\begin{equation*}
Z^*=\argmin_{z\in Z_{1:n}}\left\|\sqrt{1-2^{-2q}}z-y\right\|^2.
\end{equation*}
If $n=2^{qt}$, then
\begin{equation*}
\mathbb E\left\|\sqrt{1-2^{-2q}}Z^*-y\right\|^2\leq 2^{-2q}\left(1+\nu(t)\right)+2e^{-2t}
\end{equation*}
where 
\begin{equation*}
\nu(t)= \frac{6\log t+7}{t-6\log t-7}.
\end{equation*}
\end{lemma}

Observe that
\begin{align*}
A_{k,1}&=\left\|\hat\theta_{(k)}-\lambda_k S_k Z_{(k)}\right\|^2\\
&=\left\|\lambda_k S_k\sqrt{1-2^{-2\tilde b_k}} \check Z_{(k)}-\lambda_k S_kZ_{(k)}\right\|^2\\
&=\lambda_k^2 S_k^2\left\|\sqrt{1-2^{-2\tilde b_k}} \check Z_{(k)}-Z_{(k)}\right\|^2.
\end{align*}
Then, it follows as a result of Lemma \ref{lem_ratedistortion} that 
\begin{align*}
\mathbb E \left(A_{k,1}\given Y_{(k)}\right)&\leq \frac{(S_k^2-T_k\eps^2)^2}{S_k^2}\left(2^{-2\tilde b_k}(1+\nu_\eps)+2e^{-2T_k}\right)\\
&\leq \frac{(S_k^2-T_k\eps^2)^2}{S_k^2}\left(2^{-2\tilde b_k}(1+\nu_\eps)+2e^{-2T_1}\right)\\
&\leq \frac{(S_k^2-T_k\eps^2)^2}{S_k^2}2^{-2\tilde b_k}(1+\nu_\eps)+\frac{2c^2}{(j_k\pi)^{2m}}\eps^2,
\end{align*}
where $\nu_\eps=\frac{6\log T_1+7}{T_1-6\log T_1-7}$.
Since $A_{k,2}$ only depends on $Y_{(k)}$, $\mathbb E \left(A_{k,2}\given Y_{(k)}\right)=A_{k,2}$. Next we consider the cross term $A_{k,3}$. Write $\gamma_k=\frac{\langle\theta_{(k)},Y_{(k)}\rangle}{\|Y_{(k)}\|^2}$ and
\begin{align*}
A_{k,3}&=2\left\langle\hat\theta_{(k)}-\lambda_k S_kZ_{(k)},\lambda_k S_kZ_{(k)}-\theta_{(k)}\right\rangle\\
&=\underbrace{2\left\langle\hat\theta_{(k)}-\lambda_k S_kZ_{(k)},\gamma_kY_{(k)}-\theta_{(k)}\right\rangle}_{A_{k,3a}}\nonumber\\
&\hspace{1in}+\underbrace{2\left\langle\hat\theta_{(k)}-\lambda_k S_kZ_{(k)},\lambda_k S_kZ_{(k)}-\gamma_kY_{(k)}\right\rangle}_{A_{k,3b}}.
\end{align*}
The quantity $\gamma_k$ is chosen such that $\langle Y_{(k)},\gamma_kY_{(k)}-\theta_{(k)}\rangle=0$ and therefore
\begin{align*}
A_{k,3a}&=2\left\langle\hat\theta_{(k)}-\lambda_k S_kZ_{(k)},\gamma_kY_{(k)}-\theta_{(k)}\right\rangle\\
&=2\left\langle\Pi_{Y_{(k)}^\perp}(\hat\theta_{(k)}-\lambda_k S_kZ_{(k)}),\gamma_kY_{(k)}-\theta_{(k)}\right\rangle
\end{align*}
where $\Pi_{Y_{(k)}^\perp}$ denotes the projection onto the orthogonal complement of $Y_{(k)}$.
Due to the choice of $\check Z_{(k)}$, the projection $\Pi_{Y_{(k)}^\perp}(\hat\theta_{(k)}-\lambda_k S_kZ_{(k)})$ is rotation symmetric and hence $\mathbb E\left(A_{k,3a}\given Y_{(k)}\right)=0$. Finally, for $A_{k,3b}$ we have
\begin{align*}
&\hspace{-0.3in}\mathbb E\left(A_{k,3b}\given Y_{(k)}\right)\nonumber\\
&\leq 2\|\lambda_k S_kZ_{(k)}-\gamma_kY_{(k)}\|\mathbb E\left(\|\hat\theta_{(k)}-\lambda_k S_kZ_{(k)}\|\given Y_{(k)}\right)\\
&\leq 2\|\lambda_k S_kZ_{(k)}-\gamma_kY_{(k)}\|\sqrt{\mathbb E\left(\|\hat\theta_{(k)}-\lambda_k S_kZ_{(k)}\|^2\given Y_{(k)}\right)}\\
&\leq 2\|\lambda_k S_kZ_{(k)}-\gamma_kY_{(k)}\|\sqrt{\frac{(S_k^2-T_k\eps^2)^2}{S_k^2}2^{-2\tilde b_k}(1+\nu_\eps)+\frac{2c^2}{(j_k\pi)^{2m}}\eps^2}.
\end{align*}
Combining all the analyses above, we have
\begin{equation*}
\begin{aligned}
&\mathbb E\left(A_k\given Y_{(k)}\right)\\
&\leq \frac{(S_k^2-T_k\eps^2)^2}{S_k^2}2^{-2\tilde b_k}(1+\nu_\eps)+\frac{2c^2}{(j_k\pi)^{2m}}\eps^2+\|\lambda_k S_kZ_{(k)}-\theta_{(k)}\|^2\\
&\hspace{0.2in}+2\|\lambda_k S_kZ_{(k)}-\gamma_kY_{(k)}\|\sqrt{\frac{(S_k^2-T_k\eps^2)^2}{S_k^2}2^{-2\tilde b_k}(1+\nu_\eps)+\frac{2c^2}{(j_k\pi)^{2m}}\eps^2},
\end{aligned}
\end{equation*}
and summing over $k$ we get
\begin{equation}
\begin{aligned}
&\mathbb E\left(\|\check\theta-\theta\|^2\given Y\right)\\
&\leq \sum_{k=1}^K\frac{(S_k^2-T_k\eps^2)^2}{S_k^2}2^{-2\tilde b_k}(1+\nu_\eps)+\sum_{k=1}^K\|\lambda_k S_kZ_{(k)}-\theta_{(k)}\|^2\\
&\hspace{0.2in}+2\sum_{k=1}^K\|\lambda_k S_kZ_{(k)}-\gamma_kY_{(k)}\|\sqrt{\frac{(S_k^2-T_k\eps^2)^2}{S_k^2}2^{-2\tilde b_k}(1+\nu_\eps)+O(\eps^2)}+O(K\eps^2).
\end{aligned}\label{eqn_step2summary}
\end{equation}

\subsection*{Step 3. Expectation over data}
First we will state three lemmas, which bound the deviation of the expectation of some particular functions of the norm of a Gaussian vector to the desired quantities.  The proofs are given in Section \ref{sec_lemmaproof}.
\begin{lemma}\label{lem_oracle1}
Suppose that $X_i\sim \mathcal N(\theta_i,\sigma^2)$ independently for $i=1,\dots,n$, where $\|\theta\|^2\leq c^2$. Let $S$ be given by
\begin{equation*}
S=\begin{cases}
\sqrt{n\sigma^2}& \text{if }\|X\|<\sqrt{n\sigma^2}\\
\sqrt{n\sigma^2}+c & \text{if }\|X\|>\sqrt{n\sigma^2}+c\\
\|X\|& \text{otherwise}.
\end{cases}
\end{equation*}
Then there exists some absolute constant $C_0$ such that
\begin{equation*}
\mathbb E\left(\frac{S^2-n\sigma^2}{S}-\frac{\langle\theta,X\rangle}{\|X\|}\right)^2\leq C_0\sigma^2.
\end{equation*}
\end{lemma}
\begin{lemma}\label{lem_oracle2}
Let $X$ and $S$ be the same as defined in Lemma \ref{lem_oracle1}.
Then for $n>4$
\begin{equation*}
\mathbb E\frac{(S^2-n\sigma^2)^2}{S^2}\leq \frac{\|\theta\|^4}{\|\theta\|^2+n\sigma^2}+\frac{4n}{n-4}\sigma^2.
\end{equation*}
\end{lemma}
\begin{lemma}\label{lem_oracle3}
Let $X$ and $S$ be the same as defined in Lemma \ref{lem_oracle1}. Define 
\begin{equation*}
\hat\theta_+ =\left(\frac{\|X\|^2-n\sigma^2}{\|X\|^2}\right)_+X,\quad \hat\theta_\dag = \frac{S^2-n\sigma^2}{S\|X\|}X.
\end{equation*}
Then
\begin{equation*}
\mathbb E\|\hat\theta_\dag-\theta\|^2\leq\mathbb E\|\hat\theta_+-\theta\|^2\leq \frac{n\sigma^2\|\theta\|^2}{\|\theta\|^2+n\sigma^2}+4\sigma^2.
\end{equation*}
\end{lemma}

We now take the expectation with respect to the data on both sides of
\eqref{eqn_step2summary}. First, by the Cauchy-Schwarz inequality
\begin{equation}
\begin{aligned}
&\mathbb E\left(\|\lambda_k S_kZ_{(k)}-\gamma_kY_{(k)}\|\sqrt{\frac{(S_k^2-T_k\eps^2)^2}{S_k^2}2^{-2\tilde b_k}(1+\nu_\eps)+O(\eps^2)}\right)\\
&\leq \sqrt{\mathbb E\|\lambda_k S_kZ_{(k)}-\gamma_kY_{(k)}\|^2}\sqrt{\mathbb E\left({\frac{(S_k^2-T_k\eps^2)^2}{S_k^2}2^{-2\tilde b_k}(1+\nu_\eps)+O(\eps^2})\right)}.
\end{aligned}\label{eqn_crosstermcauchy}
\end{equation}
We then calculate
\begin{align*}
&\hspace{-0.2in}\sqrt{\mathbb E\|\lambda_k S_kZ_{(k)}-\gamma_kY_{(k)}\|^2}\nonumber\\
&=\sqrt{\mathbb E\left\|\frac{S_k^2-T_k\eps^2}{S_k}\frac{Y_{(k)}}{\|Y_{(k)}\|}-\frac{\langle\theta_{(k)},Y_{(k)}\rangle}{\|Y_{(k)}\|}\frac{Y_{(k)}}{\|Y_{(k)}\|}\right\|^2}\\
&=\sqrt{\mathbb E\left(\frac{S_k^2-T_k\eps^2}{S_k}-\frac{\langle\theta_{(k)},Y_{(k)}\rangle}{\|Y_{(k)}\|}\right)^2}\\
&\leq C_0\eps,
\end{align*}
where the last inequality is due to Lemma \ref{lem_oracle1}, and $C_0$ is the constant therein.
Plugging this in \eqref{eqn_crosstermcauchy} and summing over $k$, we get
\begin{align*}
&\hspace{-.5in}\sum_{k=1}^K\mathbb E\left(\|\lambda_k S_kZ_{(k)}-\gamma_kY_{(k)}\|\sqrt{\frac{(S_k^2-T_k\eps^2)^2}{S_k^2}2^{-2\tilde b_k}(1+\nu_\eps)+O(\eps^2)}\right)\nonumber\\
&\leq C_0\eps\sum_{k=1}^K\sqrt{\mathbb E\left({\frac{(S_k^2-T_k\eps^2)^2}{S_k^2}2^{-2\tilde b_k}(1+\nu_\eps)+O(\eps^2)}\right)}\\
&\leq C_0\sqrt{K}\eps\sqrt{\mathbb E\sum_{k=1}^K\frac{(S_k^2-T_k\eps^2)^2}{S_k^2}2^{-2\tilde b_k}(1+\nu_\eps)+O(K\eps^2)}.
\end{align*}
Therefore,
\begin{equation*}
\begin{aligned}
&\mathbb E\|\check\theta-\theta\|^2\\
&\leq \underbrace{\mathbb E\sum_{k=1}^K\frac{(S_k^2-T_k\eps^2)^2}{S_k^2}2^{-2\tilde b_k}}_{B_1}(1+\nu_\eps)+\underbrace{\mathbb E\sum_{k=1}^K\|\lambda_k S_kZ_{(k)}-\theta_{(k)}\|^2}_{B_2}\\
&\hspace{0.2in}+C_0\sqrt{K}\eps\sqrt{\mathbb E\sum_{k=1}^K\frac{(S_k^2-T_k\eps^2)^2}{S_k^2}2^{-2\tilde b_k}(1+\nu_\eps)+O(K\eps^2)}\\
&\hspace{0.2in}+O(K\eps^2).
\end{aligned}
\end{equation*}
Now we deal with the term $B_1$. Recall that the sequence $\tilde b$ solves problem \eqref{allocationproblem}, so for any sequence $b\in\Pi_{\text{blk}}$
\begin{equation*}
\sum_{k=1}^K\frac{(\check S_k^2-T_k\eps^2)^2}{\check S_k^2}2^{-2\tilde b_k}\leq \sum_{k=1}^K\frac{(\check S_k^2-T_k\eps^2)^2}{\check S_k^2}2^{-2 \bar b_k}.
\end{equation*}
Notice that 
\begin{align*}
&\left|\frac{(\check S_k^2-T_k\eps^2)^2}{\check S_k^2}-\frac{(S_k^2-T_k\eps^2)^2}{S_k^2}\right|=\left|\check S_k^2-S_k^2\right|\left|\frac{\check S_k^2S_k^2-T_k\eps^2}{\check S_k^2S_k^2}\right|=O(\eps^2)
\end{align*}
and thus,
\begin{equation*}
\sum_{k=1}^K\frac{(S_k^2-T_k\eps^2)^2}{S_k^2}2^{-2\tilde b_k}\leq \sum_{k=1}^K\frac{(S_k^2-T_k\eps^2)^2}{S_k^2}2^{-2 \bar b_k}+O(K\eps^2).
\end{equation*}
Taking the expectation, we get
\begin{equation*}
\mathbb E\sum_{k=1}^K\frac{(S_k^2-T_k\eps^2)^2}{S_k^2}2^{-2\tilde b_k}\leq \sum_{k=1}^K\mathbb E\frac{(S_k^2-T_k\eps^2)^2}{S_k^2}2^{-2\bar b_k}+O(K\eps^2).
\end{equation*}
Applying Lemma \ref{lem_oracle2}, we get for $T_k>4$
\begin{equation*}
\mathbb E\frac{(S_k^2-T_k\eps^2)^2}{S_k^2}\leq \frac{\|\theta_{(k)}\|^4}{\|\theta_{(k)}\|^2+T_k\eps^2}+\frac{4T_k}{T_k-4}\eps^2
\end{equation*}
and it follows that
\begin{equation*}
\mathbb E\sum_{k=1}^K\frac{(S_k^2-T_k\eps^2)^2}{S_k^2}2^{-2\tilde b_k}
\leq \sum_{k=1}^K\frac{\|\theta_{(k)}\|^4}{\|\theta_{(k)}\|^2+T_k\eps^2}2^{-2\bar b_k}+O(K\eps^2).
\end{equation*}
Since $b\in\Pi_{\text{blk}}$ is arbitrary,
\begin{align*}
\mathbb E\sum_{k=1}^K\frac{(S_k^2-T_k\eps^2)^2}{S_k^2}2^{-2\tilde b_k}&\leq \min_{b\in\Pi_{\text{blk}}}\sum_{k=1}^K\frac{\|\theta_{(k)}\|^4}{\|\theta_{(k)}\|^2+T_k\eps^2}2^{-2\bar b_k}+O(K\eps^2).
\end{align*}
Turning to the term $B_2$, as a result of Lemma \ref{lem_oracle3} we have
\begin{align*}
\|\lambda_k S_kZ_{(k)}-\theta_{(k)}\|^2\leq \frac{\|\theta_{(k)}\|^2T_k\eps^2}{\|\theta_{(k)}\|^2+T_k\eps^2}+4\eps^2.
\end{align*}
Combining the above results, we have shown that
\begin{align}
\mathbb E\|\check\theta-\theta\|^2\leq M+O(K\eps^2)+C_0\sqrt{K}\eps\sqrt{M+O(K\eps^2)}\label{eqn_upperboundbyM}
\end{align}
where 
\begin{align*}
M&=(1+\nu_\eps)\min_{b\in\Pi_{\text{blk}}(B)}\sum_{k=1}^K\frac{\|\theta_{(k)}\|^4}{\|\theta_{(k)}\|^2+T_k\varepsilon^2} 2^{-2 \bar b_k}+\sum_{k=1}^K\frac{\|\theta_{(k)}\|^2T_k\varepsilon^2}{\|\theta_{(k)}\|^2+T_k\varepsilon^2}\\
&=(1+\nu_\eps)\min_{b\in\Pi_{\text{blk}}(B)}\sum_{k=1}^K\frac{\|\theta_{(k)}\|^4}{\|\theta_{(k)}\|^2+T_k\varepsilon^2}2^{-2
  \bar b_k}\\
&\qquad{} +\min_{\omega\in\Omega_{\text{blk}}}\sum_{k=1}^K\left((1-\bar\omega_k)^2\|\theta_{(k)}\|^2+\bar\omega_k^2T_k\eps^2\right).
\end{align*}

\subsection*{Step 4. Blockwise constant is almost optimal} 
We now show that in terms of both bit allocation and weight assignment, blockwise constant is almost optimal. 
Let's first consider bit allocation. 
Let $B'=\frac{1}{1+3\rho_\eps}(B-T_1b_{\max})$. We are going to show that 
\begin{equation}
\min_{b\in\Pi_{\text{blk}}(B)}\sum_{k=1}^K\frac{\|\theta_{(k)}\|^4}{\|\theta_{(k)}\|^2+T_k\varepsilon^2}2^{-2 \bar b_k}\leq \min_{b\in\Pi_{\text{mon}}(B')}\sum_{j=1}^N\frac{\theta_j^4}{\theta_j^2+\eps^2}2^{-2 b_j}.\label{eqn_blkleqmon1}
\end{equation}
In fact, suppose that $b^*\in\Pi_{\text{mon}}(B')$ achieves the minimum on the right hand side, and define $b^\star$ by 
\begin{equation*}
b^\star_j=\begin{cases}
\max_{i\in B_k}b^*_i &j\in B_k\\
0 &j\geq N
\end{cases}.
\end{equation*}
The sum of the elements in $b^\star$ then satisfies
\begin{align*}
\sum_{j=1}^\infty b^\star_j&=\sum_{k=0}^{K-1}T_{k+1}\max_{j\in B_{k+1}}b^*_j\\
&=T_1b^\star_1+\sum_{k=1}^{K-1}T_{k+1}\max_{j\in B_{k+1}}b^*_j\\
&\leq T_1b_{\max}+\sum_{k=1}^{K-1}\frac{T_{k+1}}{T_k}\sum_{j\in B_k}b^*_j\\
&\leq T_1b_{\max}+(1+3\rho_\eps)\sum_{k=1}^{K-1}\sum_{j\in B_k}b^*_j\\
&\leq T_1b_{\max}+(1+3\rho_\eps)B'\\
&=B,
\end{align*}
which means that $b^\star\in\Pi_{\text{blk}}(B)$. It then follows that
\begin{align*}
&\min_{b\in\Pi_{\text{blk}}(B)}\sum_{k=1}^K\frac{\|\theta_{(k)}\|^4}{\|\theta_{(k)}\|^2+T_k\varepsilon^2}2^{-2 \bar b_k}\\
&\leq \sum_{k=1}^K\frac{\|\theta_{(k)}\|^4}{\|\theta_{(k)}\|^2+T_k\varepsilon^2}2^{-2 \bar b^\star_k}\\
&\leq \sum_{k=1}^K\sum_{j\in B_k}\frac{\theta_j^4}{\theta_j^2+\eps^2}2^{-2 b^\star_j}\numberthis\label{eqn_jensen}\\
&= \sum_{j=1}^N\frac{\theta_j^4}{\theta_j^2+\eps^2}2^{-2 b^*_j}\\
&= \min_{b\in\Pi_{\text{mon}}(B')}\sum_{j=1}^N\frac{\theta_j^4}{\theta_j^2+\eps^2}2^{-2 b_j},
\end{align*}
where \eqref{eqn_jensen} is due to Jensen's inequality on the convex function $\frac{x^2}{x+\eps^2}$
\begin{equation*}
\frac{\left(\frac{1}{T_k}\|\theta_{(k)}\|^2\right)^2}{\frac{1}{T_k}\|\theta_{(k)}\|^2+\eps^2}\leq \frac{1}{T_k}\sum_{j\in B_k}\frac{\theta_j^4}{\theta_j^2+\eps^2}.
\end{equation*}
Next, for the weights assignment, by Lemma 3.11 in \cite{tsybakov2008introduction}, we have
\begin{equation}\begin{aligned}
\min_{\omega\in\Omega_{\text{blk}}}\sum_{k=1}^K&\left((1-\bar\omega_k)^2\|\theta_{(k)}\|^2+\bar\omega_k^2T_k\eps^2\right)\\
&\leq (1+3\rho_\eps)\left(\min_{\omega\in\Omega_{\text{mon}}}\sum_{k=1}^K\left((1-\omega_j)^2\theta_j^2+\omega_j^2\eps^2\right)\right)+T_1\eps^2.\label{eqn_blkleqmon2}
\end{aligned}
\end{equation}
Combining \eqref{eqn_blkleqmon1} and \eqref{eqn_blkleqmon2}, we get
\begin{align*}
M&=(1+\nu_\eps)\min_{b\in\Pi_{\text{blk}}(B)}\sum_{k=1}^K\frac{\|\theta_{(k)}\|^4}{\|\theta_{(k)}\|^2+T_k\varepsilon^2}2^{-2 \bar b_k}\\
&\hspace{.3in}+\min_{\omega\in\Omega_{\text{blk}}}\sum_{k=1}^K\left((1-\bar\omega_k)^2\|\theta_{(k)}\|^2+\bar\omega_k^2T_k\eps^2\right)\nonumber\\
&\leq (1+\nu_\eps)\min_{b\in\Pi_{\text{blk}}(B)}\sum_{k=1}^K\frac{\|\theta_{(k)}\|^4}{\|\theta_{(k)}\|^2+T_k\varepsilon^2}2^{-2 \bar b_k}\\
&\hspace{.3in}+(1+3\rho_\eps)\min_{\omega\in\Omega_{\text{mon}}}\sum_{k=1}^K\left((1-\bar\omega_k)^2\|\theta_{(k)}\|^2+\bar\omega_k^2T_k\eps^2\right)+T_1\eps^2\nonumber\\
&\leq (1+\nu_\eps)\bigg(\min_{b\in\Pi_{\text{mon}}(B')}\sum_{j=1}^N\frac{\theta_j^4}{\theta_j^2+\eps^2}2^{-2 b_j}\\
&\hspace{.3in}+\min_{\omega\in\Omega_{\text{mon}}}\sum_{j=1}^N\left((1-\omega_j)^2\theta_j^2+\omega_j^2\eps^2\right)\bigg)+T_1\eps^2.\nonumber
\end{align*}
Then by Lemma \ref{lem_equivalence}, 
\begin{equation*}
M\leq (1+\nu_\eps)V_\varepsilon(m,c,B')+T_1\eps^2.
\end{equation*}
which, plugged into \eqref{eqn_upperboundbyM}, gives us
\begin{equation*}\begin{aligned}
&\mathbb E\|\check\theta-\theta\|^2\leq (1+\nu_\eps)V_\varepsilon(m,c,B')+O(K\eps^2)\\
&\hspace{1in}+C_0\sqrt{K}\eps\sqrt{(1+\nu_\eps)V_\varepsilon(m,c,B')+O(K\eps^2)}.
\end{aligned}\end{equation*}
Recall that
\begin{equation*}
\nu_\eps=O\left(\frac{\log\log(1/\eps)}{\log(1/\eps)}\right),\quad K=O(\log^2(1/\eps)),
\end{equation*}
and that 
\begin{equation*}
\lim_{\eps\to0}\frac{B'}{B}=\lim_{\eps\to 0}\frac{1}{1+3\rho_\eps}\left(1-\frac{T_1b_{\max}}{B}\right)=1.
\end{equation*}
Thus,
\begin{equation*}
\lim_{\eps\to 0}\frac{V_\eps(m, c, B')}{V_\eps( m, c,B)}=1.
\end{equation*}
Also notice that no matter how $B$ grows as $\eps\to0$, $V_\eps(m,c,B)=O(\eps^{\frac{4m}{2m+1}})$. Therefore, 
\begin{align*}
&\hspace{-0.5in}\lim_{\eps\to0}\frac{\mathbb E\|\check\theta-\theta\|^2}{V_\eps(B,m,c)}\nonumber\\
&\leq \lim_{\eps\to0}\Bigg((1+\nu_\eps)\frac{V_\eps(B',m,c)}{V_\eps(B,m,c)}+\frac{O(K\eps^2)}{V(B,m,c)}\\
&\hspace{0.2in}+C_0\sqrt{(1+\nu_\eps)\frac{K\eps^2}{V_\eps(B,m,c)}\frac{V_\eps(B',m,c)}{V_\eps(B,m,c)}+\left(\frac{O(K\eps^2)}{V_\eps(B,m,c)}\right)^2}\Bigg)\nonumber\\
&=1
\end{align*}
which concludes the proof.
\begin{lemma}\label{lem_equivalence}
Let $V_1$ be the value of the optimization
\begin{equation*}
\tag{$A_1$}\label{eqn_A1}
\begin{aligned}
\max_{\theta}\min_b\ &\ \sum_{j=1}^N\left(\frac{\theta_j^4}{\theta_j^2+\eps^2}2^{-2b_j}+\frac{\theta_j^2\eps^2}{\theta_j^2+\eps^2}\right)\\
\text{such that}\ &\ \sum_{j=1}^Nb_j\leq B,\ b_j\geq0,\ \sum_{j=1}^J a_j^2\theta_j^2\leq \frac{c^2}{\pi^{2m}},
\end{aligned}
\end{equation*}
and let $V_2$ be the value of the optimization
\begin{equation*}
\tag{$A_2$}\label{eqn_A2}
\begin{aligned}
\max_{\theta}\min_{b,\omega}\ &\ \sum_{j=1}^N\left(\frac{\theta_j^4}{\theta_j^2+\eps^2}2^{-2b_j}+(1-\omega_j)^2\theta_j^2+\omega_j^2\eps^2\right)\\
\text{such that}\ &\ \sum_{j=1}^Nb_j\leq B,\ b_{j-1}\geq b_j,\ 0\leq b_j\leq b_{\max},\ \omega_{j-1}\geq\omega_j,\\
&\ \sum_{j=1}^J a_j^2\theta_j^2\leq \frac{c^2}{\pi^{2m}}.
\end{aligned}
\end{equation*}
Then $V_1=V_2$.
\end{lemma}

\subsection{Proofs of Lemmas}\label{sec_lemmaproof}

\begin{proof}[Proof of Lemma \ref{lem_ratedistortion}] Let $\zeta(t)$ be a positive function of $t$ to be specified later.  Let
\begin{equation*}
p_0=\mathbb P\left(\left\|\sqrt{1-2^{-2q}}Z_1-y\right\|\leq 2^{-q}\sqrt{1+\zeta(t)}\right).
\end{equation*}
By Lemma \ref{lem:sphere}, when $\zeta(t)\leq 2(1-2^{-2q})$, $p_0$ can be lower bounded by 
\begin{equation*}
p_0\geq\frac{\Gamma(\frac{t}{2}+1)}{\sqrt{\pi} t\Gamma(\frac{t+1}{2})}\left(2^{-q}\sqrt{1+\zeta(t)/2}\right)^{t-1}.
\end{equation*}
We obtain that
\begin{align*}
&\hspace{-.3in}\mathbb E\left\|\sqrt{1-2^{-2q}}Z^*-y\right\|^2\nonumber\\
&\leq 2^{-2q}(1+\zeta(t))+2\mathbb P\left(\left\|\sqrt{1-2^{-2q}}Z^*-y\right\|> 2^{-q}\sqrt{1+\zeta(t)}\right)\\
&= 2^{-2q}(1+\zeta(t))+2(1-{p_0})^n.
\end{align*}
To upper bound $(1-{p_0})^n$, we consider
\begin{align*}
\log\left((1-{p_0})^n\right)&=n\log(1-p_0)\leq -np_0\\
&\leq -2^{qt}\frac{\Gamma(\frac{t}{2}+1)}{\sqrt{\pi} t\Gamma(\frac{t+1}{2})}\left(2^{-q}\sqrt{1+\zeta(t)/2}\right)^{t-1}\\
&\leq -2^q\frac{\Gamma(\frac{t}{2}+1)}{\sqrt{\pi} t\Gamma(\frac{t+1}{2})}(1+\zeta(t)/2)^{(2/\zeta(t)+1)\frac{t-1}{2(2/\zeta(t)+1)}}\\
&\leq -\frac{\sqrt{2\pi}(\frac{t}{2})^{\frac{t}{2}+\frac{1}{2}}e^{-\frac{t}{2}}}{\sqrt{\pi} te(\frac{t}{2}-\frac{1}{2})^{\frac{t}{2}}e^{-(\frac{t}{2}-\frac{1}{2})}}e^{\frac{t-1}{2(2/\zeta(t)+1)}}\\
&=-e^{-\frac{3}{2}}t^{-\frac{1}{2}}\left(\frac{t}{t-1}\right)^{\frac{t}{2}}e^{\frac{t-1}{2(2/\zeta(t)+1)}}\\
&\leq -e^{-1}t^{-\frac{1}{2}}e^{\frac{t-1}{2(2/\zeta(t)+1)}}
\end{align*}
where we have used Stirling's approximation in the form
\begin{equation*}
\sqrt{2\pi}z^{z+1/2}e^{-z}\leq \Gamma(z+1)\leq ez^{z+1/2}e^{-z}.
\end{equation*}
In order for $(1-p_0)^n\leq e^{-2t}$ to hold, we need
\begin{equation*}
-2t = -e^{-1}t^{-\frac{1}{2}}e^{\frac{t-1}{2(2/\zeta(t)+1)}},
\end{equation*}
which leads to the choice of $\zeta(t)$
\begin{equation*}
\zeta(t)= \frac{2}{\frac{t-1}{2\log(2et^{\frac{3}{2}})}-1}=\frac{6\log t+4\log(2e)}{t-3\log t-2\log(2e)-1}.
\end{equation*}
Thus, we have shown that when $q$ is not too close to 0, satisfying $1-2^{-2q} \geq \zeta(t)/2$,
we have
\[
\mathbb E\left\|\sqrt{1-2^{-2q}}Z^*-y\right\|^2 \leq 2^{-2q}\left(1+\zeta(t)\right) + e^{-2t}.
\]
When $1-2^{-2q} < \zeta(t)/2$, we observe that
\begin{align*}
\mathbb E\left\|\sqrt{1-2^{-2q}}Z^*-y\right\|^2 &= 1-2^{-2q} + 1 -2\sqrt{1-2^{-2q}}\,\mathbb E\langle Z^*,y\rangle \\
&\leq 2-2^{-2q} = 2^{-2q}\left(1+2\left(2^{2q}-1\right)\right)
\end{align*}
and that
\[
2(2^{2q}-1) < \frac{2}{1-\zeta(t)/2}-2= \frac{2\zeta(t)}{2-\zeta(t)} = \frac{6\log t+4\log(2e)}{t-6\log t-4\log(2e)-1}.
\]
Now take $\nu(t) = \frac{6\log t+7}{t-6\log t-7}$. Notice that $\nu(t)>\frac{6\log t+4\log(2e)}{t-6\log t-4\log(2e)-1}\geq \zeta(t)$, we have for any $q\geq 0$
\[
\mathbb E\left\|\sqrt{1-2^{-2q}}Z^*-y\right\|^2 \leq 2^{-2q}\left(1+\nu(t)\right) + e^{-2t}.
\]
\end{proof}

\begin{lemma}\label{lem:sphere}
Suppose $Z$ is a $t$-dimensional random vector uniformly distributed on the unit sphere $\mathbb S^{t-1}$. 
Let $y$ be a fixed vector on the unit sphere. For $\delta<1$ and $\zeta>0$ satisfying $\zeta\leq 2(1-\delta^2)$, define
\[
p_0 = \mathbb P\left(\|Z-y\|\leq \delta\sqrt{1+\zeta}\right).
\]
We have
\[
p_0 \geq \frac{\Gamma\left(\frac{t}{2}+1\right)}{\sqrt{\pi} t\Gamma\left(\frac{t+1}{2}\right)}\left(\delta\sqrt{1+\zeta/2}\right)^{t-1}
\]
\end{lemma}
\begin{proof}
The proof is based on an idea from \cite{sakrison1968geometric}.
\begin{figure}
\begin{center}
\includegraphics{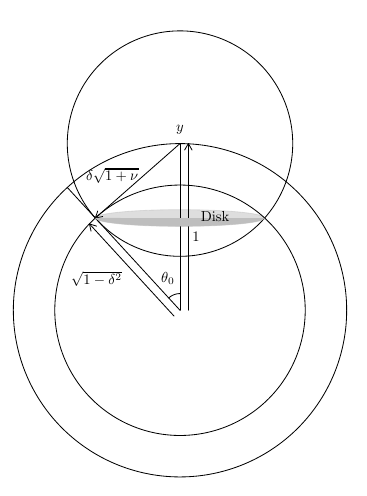} \\
\caption{Illustration of the geometry for calculating $p_0$}\label{fig:geometry}
\end{center}
\end{figure}
Denote by $V_t$ and $A_t$ the volume and the surface area of a $t$-dimensional unit sphere, respectively.
We have
\[
V_t = \int_0^1 A_tr^{t-1}dr = \frac{1}{t}A_t.
\]
From the geometry of the situation as illustrated in Figure \ref{fig:geometry}, $p_0$ is equal to 
the ratio of two areas $S_1$ and $S_2$.
The first area $S_1$ is the portion of the surface area of the sphere of radius $\sqrt{1-\delta^2}$ and center $O$
contained within the sphere of radius $\delta\sqrt{1+\zeta}$ and center $y$.
It is the surface area of a $(t-1)$-dimensional polar cap of radius $\sqrt{1-\delta^2}$ and polar angle $\theta_0$,
and can be lower bounded by the area of a $(t-1)$-dimensional disk of radius $\sqrt{1-\delta^2}\sin\theta_0$, that is,
\[
S_1\geq V_{t-1}\left(\sqrt{1-\delta^2}\sin\theta_0\right)^{t-1} = \frac{1}{t-1}A_{t-1}\left(\sqrt{1-\delta^2}\sin\theta_0\right)^{t-1}
\]
The second area $S_2$ is simply the surface area of a $(t-1)$-dimensional sphere of radius $\sqrt{1-\delta^2}$
\[
S_2 = A_t \left(\sqrt{1-\delta^2}\right)^{t-1}.
\]
Therefore, we obtain
\begin{align*}
p_0 &= \frac{S_1}{S_2} \geq \frac{\frac{1}{t-1}A_{t-1}\left(\sqrt{1-\delta^2}\sin\theta_0\right)^{t-1}}{A_t \left(\sqrt{1-\delta^2}\right)^{t-1}}=\frac{A_{t-1}}{(t-1)A_t}\left(\sin\theta_0\right)^{t-1} = \frac{\Gamma\left(\frac{t+1}{2}+\frac{1}{2}\right)}{\sqrt{\pi} t\Gamma\left(\frac{t+1}{2}\right)}\left(\sin\theta_0\right)^{t-1},
\end{align*}
where we have used the well-known relationship between $A_{t-1}$ and $A_t$
\[
\frac{A_{t-1}}{A_t} = \frac{1}{\sqrt{\pi}}\frac{(t-1)\Gamma\left(\frac{t}{2}+1\right)}{t\Gamma\left(\frac{t-1}{2}+1\right)}.
\]
Now we need to calculate $\sin\theta_0$. By the law of cosines, we have
\[
\cos\theta_0 = \frac{1+1-\delta^2-\delta^2(1+\zeta)}{2\sqrt{1-\delta^2}} = \frac{1-\delta^2(1+\zeta/2)}{\sqrt{1-\delta}^2}
\]
and it follows that
\[
\sin^2\theta_0 = 1-\cos^2\theta_0 = 1 - \frac{1+\delta^4(1+\zeta/2)^2-2\delta^2(1+\zeta/2)}{1-\delta^2}
= \delta^2(1+\zeta) - \frac{\delta^4\zeta^2}{4(1-\delta^2)}.
\]
Now since $\zeta\leq 2(1-\delta^2)$, we get
\[
\sin\theta_0\geq \delta\sqrt{1+\zeta/2},
\]
which completes the proof. 
\end{proof}

\begin{proof}[Proof of Lemma \ref{lem_oracle1}]
We first claim that
\begin{equation*}
\mathbb E\left(\frac{S^2-n\sigma^2}{S}-\frac{\langle\theta,X\rangle}{\|X\|}\right)^2\leq\mathbb E\left(\frac{\|X\|^2-n\sigma^2}{\|X\|}-\frac{\langle\theta,X\rangle}{\|X\|}\right)^2.\label{eqn_oracle1claim}
\end{equation*}
In fact, writing $\mathbb E_r(\cdot)$ for the conditional expectation $\mathbb E(\cdot\given \|X\|=r)$, it suffices to show that for $r<\sqrt{n\sigma^2}$ and $r>\sqrt{n\sigma^2}+c$  
\begin{equation*}
\mathbb E_r\left(\frac{S^2-n\sigma^2}{S}-\frac{\langle\theta,X\rangle}{\|X\|}\right)^2\leq\mathbb E_r\left(\frac{\|X\|^2-n\sigma^2}{\|X\|}-\frac{\langle\theta,X\rangle}{\|X\|}\right)^2.
\end{equation*}
When $r<\sqrt{n\sigma^2}$, it is equivalent to
\begin{align*}
\mathbb E_r\left(\frac{\langle\theta,X\rangle}{\|X\|}\right)^2\leq\mathbb E_r\left(\frac{\langle\theta,X\rangle}{\|X\|}-\frac{\|X\|^2-n\sigma^2}{\|X\|}\right)^2
\end{align*}
It is then sufficient to show that $\mathbb E_r\langle \theta,X\rangle\geq 0$. This can be obtained by following a similar argument as in Lemma A.6 in \cite{tsybakov2008introduction}. When $r>\sqrt{n\sigma^2}+c$, we need to show that
\begin{equation*}
\mathbb E_r\left(\frac{(\sqrt{n\sigma^2}+c)^2-n\sigma^2}{\sqrt{n\sigma^2}+c}-\frac{\langle\theta,X\rangle}{\|X\|}\right)^2\leq\mathbb E_r\left(\frac{\|X\|^2-n\sigma^2}{\|X\|}-\frac{\langle\theta,X\rangle}{\|X\|}\right)^2,
\end{equation*}
which, after some algebra, boils down to
\begin{equation*}
\frac{(\sqrt{n\sigma^2}+c)^2-n\sigma^2}{\sqrt{n\sigma^2}+c}+\frac{r^2-n\sigma^2}{r}\geq \frac{2}{r}\mathbb E_r\langle \theta, X\rangle.
\end{equation*}
This holds because
\begin{align*}
&r\left(\frac{(\sqrt{n\sigma^2}+c)^2-n\sigma^2}{\sqrt{n\sigma^2}+c}+\frac{r^2-n\sigma^2}{r}- \frac{2}{r}\mathbb E_r\langle \theta, X\rangle\right)\\
&\geq \|\theta\|^2+r^2-n\sigma^2-2\mathbb E_r\langle\theta,X\rangle\\
&\geq \mathbb E_r\|X-\theta\|^2-n\sigma^2\\
&\geq 0
\end{align*}
where we have used the assumption that $r>\sqrt{n\sigma^2}+c$, $\|\theta\|\leq c$ and that
\begin{equation*}
\mathbb E_r\|X-\theta\|\geq\mathbb E_r\|X\|-\|\theta\|\geq\sqrt{n\sigma^2}.
\end{equation*}

Now that we have shown \eqref{eqn_oracle1claim} and noting that
\begin{equation*}
\mathbb E\left(\frac{\|X\|^2-n\sigma^2}{\|X\|}-\frac{\langle\theta,X\rangle}{\|X\|}\right)^2=\sigma^2\mathbb E\left(\frac{\|X/\sigma\|^2-n}{\|X/\sigma\|}-\frac{\langle\theta/\sigma,X/\sigma\rangle}{\|X/\sigma\|}\right)^2,
\end{equation*}
we can assume that $X\sim N(\theta,I_n)$ and equivalently show that there exists a universal constant $C_0$ such that
\begin{equation*}
\mathbb E\left(\frac{\|X\|^2-n}{\|X\|}-\frac{\langle\theta,X\rangle}{\|X\|}\right)^2\leq C_0
\end{equation*}
holds for any $n$ and $\theta$. Letting $Z=X-\theta$ and writing $\|\theta\|^2=\xi$, we have 
\begin{align*}
&\mathbb E\left(\frac{\|X\|^2-n}{\|X\|}-\frac{\langle\theta,X\rangle}{\|X\|}\right)^2\\
&=\mathbb E\left(\frac{\|Z+\theta\|^2-n-\xi}{\|Z+\theta\|}-\frac{\langle\theta,Z\rangle}{\|Z+\theta\|}\right)^2\\
&\leq 2\mathbb E\left(\frac{\|Z+\theta\|^2-n-\xi}{\|Z+\theta\|}\right)^2+2\mathbb E\left(\frac{\langle\theta,Z\rangle}{\|Z+\theta\|}\right)^2\\
&\leq 2\mathbb E\|Z+\theta\|^2-4(n+\xi)+2\mathbb E\frac{(n+\xi)^2}{\|Z+\theta\|^2}+2\mathbb E\left(\frac{\langle\theta,Z\rangle}{\|Z+\theta\|}\right)^2\\
&\leq2(n+\xi)-4(n+\xi)+2\frac{(n+\xi)^2}{n+\xi-4}+2\mathbb E\left(\frac{\langle\theta,Z\rangle}{\|Z+\theta\|}\right)^2\\
&=\frac{8(n+\xi)}{n+\xi-4}+2\mathbb E\left(\frac{\langle\theta,Z\rangle}{\|Z+\theta\|}\right)^2.
\end{align*}
where the last inequality is due to Lemma \ref{lem_inversechisq}.
To bound the last term, we apply the Cauchy-Schwarz inequality and get
\begin{align*}
\mathbb E\left(\frac{\langle\theta,Z\rangle}{\|Z+\theta\|}\right)^2&\leq \sqrt{\mathbb E\frac{1}{\|Z+\theta\|^4}\mathbb E\langle\theta,Z\rangle^4}\\
&\leq \sqrt{\frac{3(n-4)\xi^2}{(n-6)(n+\xi-4)(n+\xi-6)}}
\end{align*}
where the last inequality is again due to Lemma \ref{lem_inversechisq}. Thus we just need to take $C_0$ to be
\begin{equation*}
\sup_{n\geq7,\xi\geq 0}\frac{8(n+\xi)}{n+\xi-4}+2\sqrt{\frac{3(n-4)\xi^2}{(n-6)(n+\xi-4)(n+\xi-6)}},
\end{equation*}
which is apparently a finite quantity. 
\end{proof}

\begin{proof}[Proof of Lemma \ref{lem_oracle2}]
Since the function $(x^2-n\sigma^2)^2/x^2$ is decreasing on $(0,\sqrt{n\sigma^2})$ and increasing on $(\sqrt{n\sigma^2},\infty)$, we have 
\begin{equation*}
\frac{(S^2-n\sigma^2)^2}{S^2}\leq \frac{(\|X\|^2-n\sigma^2)^2}{\|X\|^2},
\end{equation*}
and it follows that if $n>4$
\begin{align}
\mathbb E\frac{(S^2-n\sigma^2)^2}{S^2}&\leq\mathbb E \frac{(\|X\|^2-n\sigma^2)^2}{\|X\|^2}\\
&=\mathbb E\|X\|^2-2n\sigma^2+n^2\sigma^4\mathbb E\left(\frac{1}{\|X\|^2}\right)\\
&\leq\|\theta\|^2-n\sigma^2+\frac{n^2\sigma^4}{\|\theta\|^2+n\sigma^2-4\sigma^2}\label{inversechisq}\\
&\leq \frac{\|\theta\|^4}{\|\theta\|^2+n\sigma^2}+\frac{4n}{n-4}\sigma^2\label{oracle1}
\end{align}
where \eqref{inversechisq} is due to Lemma \ref{lem_inversechisq}, and \eqref{oracle1} is obtained by
\begin{align*}
&\|\theta\|^2-n\sigma^2+\frac{n^2\sigma^4}{\|\theta\|^2+n\sigma^2-4\sigma^2}-\frac{\|\theta\|^4}{\|\theta\|^2+n\sigma^2}\\
&=\frac{\|\theta\|^4+4\sigma^2(n\sigma^2-\|\theta\|^2)}{\|\theta\|^2+n\sigma^2-4\sigma^2}-\frac{\|\theta\|^4}{\|\theta\|^2+n\sigma^2}\\
&=\frac{4n^2\sigma^6}{(\|\theta\|^2+n\sigma^2-4\sigma^2)(\|\theta\|^2+n\sigma^2)}\\
&\leq \frac{4n}{n-4}\sigma^2.
\end{align*}
\end{proof}

\begin{proof}[Proof of Lemma \ref{lem_oracle3}]
First, the second inequality 
\begin{equation*}
\mathbb E\|\hat\theta_+-\theta\|^2\leq \frac{n\sigma^2\|\theta\|^2}{\|\theta\|^2+n\sigma^2}+4\sigma^2
\end{equation*}
is given by Lemma 3.10 from \cite{tsybakov2008introduction}. We thus focus on the first inequality. For convenience we write
\begin{equation*}
g_+(x)=\left(\frac{\|x\|^2-n\sigma^2}{\|x\|^2}\right)_+,\quad g_\dag(x)=\frac{s(x)^2-n\sigma^2}{s(x)\|x\|}
\end{equation*}
with
\begin{equation*}
s(x)=\begin{cases}
\sqrt{n\sigma^2}& \text{if }\|x\|<\sqrt{n\sigma^2}\\
\sqrt{n\sigma^2}+c & \text{if }\|x\|>\sqrt{n\sigma^2}+c\\
\|x\|& \text{otherwise}
\end{cases}.
\end{equation*}
Notice that $g_+(x)=g_\dag(x)$ when $\|x\|\leq \sqrt{n\sigma^2}+c$ and $g_+(x)>g_\dag(x)$ when $\|x\|> \sqrt{n\sigma^2}+c$. Since $g_\dag$ and $g_+$ both only depend on $\|x\|$, we sometimes will also write $g_\dag(\|x\|)$ for $g_\dag(x)$ and $g_+(\|x\|)$ for $g_+(x)$. Setting $\mathbb E_r(\cdot)$ to denote the conditional expectation $\mathbb E(\cdot\given \|X\|=r)$ for brevity, it suffices to show that for $r\geq \sqrt{n\sigma^2}+c$
\begin{align*}
&&\mathbb E_r\left(\|g_\dag(X) X-\theta\|^2\right)&\leq\mathbb E_r\left(\|g_+(X) X-\theta\|^2\right)\\
&\Longleftrightarrow& g_\dag(r)^2r^2-2g_\dag(r)\mathbb E_r\langle X,\theta\rangle&\leq g_+(r)^2r^2-2g_+(r)\mathbb E_r\langle X,\theta\rangle\\
&\Longleftrightarrow& (g_\dag(r)^2-g_+(r)^2)r^2&\geq 2(g_\dag(r)-g_+(r))\mathbb E_r\langle X,\theta\rangle\\
&\Longleftrightarrow& (g_\dag(r)+g_+(r))r^2&\geq2\mathbb E_r\langle X,\theta\rangle.\numberthis\label{gplusg}
\end{align*}
On the other hand, we have
\begin{align*}
(g_\dag(r)+g_+(r))r^2&\geq \left(\frac{\|\theta\|^2}{r^2}+\frac{r^2-n\sigma^2}{r^2}\right)r^2\\
&=\|\theta\|^2+r^2-n\sigma^2\\
&=\|\theta\|^2+r^2-2\mathbb E_r\langle X,\theta\rangle-n\sigma^2+2\mathbb E_r\langle X,\theta\rangle\\
&=\mathbb E_r\|X-\theta\|^2-n\sigma^2+2\mathbb E_r\langle X,\theta\rangle\\
&\geq 2\mathbb E_r\langle X,\theta\rangle
\end{align*}
where the last inequality is because
\begin{equation*}
\|X-\theta\|^2\geq(\|X\|-\|\theta\|)^2\geq n\sigma^2.
\end{equation*}
Thus, \eqref{gplusg} holds and hence $\mathbb E\|\hat\theta_\dag-\theta\|^2\leq\mathbb E\|\hat\theta_+-\theta\|^2$.
\end{proof}

\begin{proof}[Proof of Lemma \ref{lem_equivalence}] It is easy to see that $V_1\leq V_2$, because for any $\theta$ the inside minimum is smaller for (\ref{eqn_A1}) than for (\ref{eqn_A2}). Next, we will show $V_1\geq V_2$.

Suppose that $\theta^*$ achieves the value $V_2$, with corresponding $b^*$ and $\omega^*$. We claim that $\theta^*$ is non-increasing. In fact, if $\theta^*$ is not non-increasing, then there must exist an index $j$ such that $\theta_j^*<\theta_{j+1}^*$ and for simplicity let's assume that $\theta^*_1<\theta^*_2$. We are going to show that this leads to $b^*_1=b^*_2$ and $\omega^*_1=\omega_2^*$. Write 
\begin{equation*}
s_1=\frac{\theta_1^{*4}}{\theta_1^{*2}+\eps^2},\quad s_2=\frac{\theta_2^{*4}}{\theta_2^{*2}+\eps^2}.
\end{equation*}
We have $s_1<s_2$. Let $\bar b^* = \frac{b_1^*+b_2^*}{2}$ and observe that $b_1^*\geq \bar b^*\geq b_2^*$. Notice that
\begin{align*}
&\left(s_12^{-2b_1^*}+s_22^{-2b_2^*}\right)-\left(s_12^{-2\bar b^*}+s_22^{-2\bar b^*}\right)\\
&=s_1\left(2^{-2b_1^*}-2^{-2\bar b^*}\right)+s_2\left(2^{-2b_2^*}-2^{-2\bar b^*}\right)\\
&\geq s_2\left(2^{-2b_1^*}-2^{-2\bar b^*}\right)+s_2\left(2^{-2b_2^*}-2^{-2\bar b^*}\right)\\
&\geq s_2\left(2^{-2b_1^*}+2^{-2b_2^*}-2\cdot 2^{-2\bar b^*}\right)\\
&\geq 0,
\end{align*}
where equality holds if and only if $b^*_1=b_2^*$, since $s_2>s_1\geq 0$. Hence, $b^*_1$ and $b_2^*$ have to be equal, or otherwise it would contradict with the assumption that $b^*$ achieves the inside minimum of (\ref{eqn_A2}). Now turn to $\omega^*$. Write $\bar \omega^*=\frac{\omega^*_1+\omega^*_2}{2}$ and note that $\omega_1^*\geq\bar\omega^*\geq\omega_2^*$. Consider
\begin{align*}
&\left((1-\omega^*_1)^2\theta_1^{*2}+\omega_1^{*2}\eps^2\right)+\left((1-\omega^*_2)^2\theta_2^{*2}+\omega_2^{*2}\eps^2\right)-\left((1-\bar\omega^*)^2(\theta_1^{*2}+\theta_2^{*2})+2\bar\omega^{*2}\eps^2\right)\\
&=\left((1-\omega_1^*)^2-(1-\bar\omega^*)^2\right)\theta_1^{*2}+\left((1-\omega_2^*)^2-(1-\bar\omega^*)^2\right)\theta_2^{*2}+\left(\omega_1^{*2}+\omega_2^{*2}-2\bar\omega^{*2}\right)\eps^2\\
&\geq \left((1-\omega_1^*)^2-(1-\bar\omega^*)^2\right)\theta_2^{*2}+\left((1-\omega_2^*)^2-(1-\bar\omega^*)^2\right)\theta_2^{*2}+\left(\omega_1^{*2}+\omega_2^{*2}-2\bar\omega^{*2}\right)\eps^2\\
&=\left((1-\omega_1^*)^2+(1-\omega_2^*)^2-2(1-\bar\omega^*)^2\right)\theta_2^{*2}+\left(\omega_1^{*2}+\omega_2^{*2}-2\bar\omega^{*2}\right)\eps^2\\
&\geq 0,
\end{align*}
where the equality holds if and only if $\omega_1^*=\omega_2^*$. Therefore, $\omega_1^*$ and $\omega_2^*$ must be equal. Now, with $b_1^*=b_2^*$ and $\omega_1^*=\omega_2^*$, we can switch $\theta_1^*$ and $\theta_2^*$ without increasing the objective function and violating the constraints. Thus, our claim that $\theta^*$ is non-increasing is justified.

Now that we have shown that the solution triplet $(\theta^*,b^*,\omega^*)$ to (\ref{eqn_A2}) satisfy that $\theta^*$ is non-increasing, 
in order to prove $V_1\geq V_2$, it suffices to show that if we take $\theta=\theta^*$ in (\ref{eqn_A1}), the minimizer $b^\star$ is non-increasing and $b^\star_1\leq b_{\max}$. In fact, if so, we will have $b^\star=b^*$ as well as $\omega^*=\frac{\theta^{*2}_j}{\theta_j^{*2}+\eps^2}$ and then
\begin{equation*}
V_1\geq\min_{b:\sum_{j=1}^Nb_j\leq B}\sum_{j=1}^N\left(\frac{\theta_j^{*4}}{\theta_j^{*2}+\eps^2}2^{-2b_j}+\frac{\theta_j^{*2}\eps^2}{\theta_j^{*2}+\eps^2}\right)\geq V_2.
\end{equation*}
Let's take $\theta=\theta^*$ in (\ref{eqn_A1}). The optimal $b^\star$ is non-increasing because the solution is given by the ``reverse water-filling'' scheme and $\theta^*$ is non-increasing. 
Next, we will show that $b^\star_1\leq b_{\max}$. If $b_1^\star>b_{\max}$, then we would have for $j=1,\dots,N$
\begin{equation*}
\frac{\theta_j^{*4}}{\theta_j^{*2}+\eps^2}2^{-2b^\star_j}\leq \frac{\theta_1^{*4}}{\theta_1^{*2}+\eps^2}2^{-2b^\star_1}\leq \theta_1^{*2}2^{-2b_{\max}}\leq c^22^{-4\log(1/\eps)}=c^2\eps^4,
\end{equation*}
where the first inequality follows from the ``reverse water-filling'' solution, 
and therefore
\begin{equation*}
\sum_{j=1}^N\frac{\theta_j^{*4}}{\theta_j^{*2}+\eps^2}2^{-2b^\star_j}\leq Nc^2\eps^4=o(\eps^{\frac{4m}{2m+1}}),
\end{equation*}
which would not give the optimal solution. 
Hence, $b_1^\star\leq b_{\max}$, and this completes the proof. 
\end{proof}

\begin{lemma}\label{lem_inversechisq}
Suppose that $W_{n,\xi}$ follows a non-central chi-square distribution with $n$ degrees of freedom and non-centrality parameter $\xi$. We have for $n\geq 5$
\begin{equation*}
\mathbb E\left(W_{n,\xi}^{-1}\right)\leq\frac{1}{n+\xi-4},\label{eqn_inversechisq1}
\end{equation*}
and for $n\geq 7$
\begin{equation*}
\mathbb E\left(W_{n,\xi}^{-2}\right)\leq\frac{n-4}{(n-6)(n+\xi-4)(n+\xi-6)}.\label{eqn_inversechisq2}
\end{equation*}
\end{lemma}
\begin{proof}
It is well known that the non-central chi-square random variable $W_{n,\xi}$ can be written as a Poisson-weighted mixture of central chi-square distributions, i.e., $W_{n,\xi}\sim\chi^2_{n+2K}$ with $K\sim\text{Poisson}(\xi/2)$. Then
\begin{align*}
\mathbb E\left(W_{n,\xi}^{-1}\right)&=\mathbb E\left(\mathbb E(W_{n,\xi}^{-1}\given K)\right)=\mathbb E\left(\frac{1}{n+2K-2}\right)\\
&\geq \frac{1}{n+2\mathbb E K-2}=\frac{1}{n+\xi-2}
\end{align*}
where we have used the fact that $\mathbb E(1/\chi^2_n)=n-2$ and Jensen's inequality. 
Similarly, we have
\begin{align*}
\mathbb E\left(W_{n,\xi}^{-2}\right)&=\mathbb E\left(\mathbb E(W_{n,\xi}^{-2}\given K)\right)=\mathbb E\left(\frac{1}{(n+2K-2)(n+2K-4)}\right)\\
&\geq \frac{1}{(n+2\mathbb E K-2)(n+2\mathbb E K-4)}\\
&=\frac{1}{(n+\xi-2)(n+\xi-4)}
\end{align*}
Using the Poisson-weighted mixture representation, the following recurrence relation can be derived \cite{chattamvelli1995recurrence}
\begin{align}
1&=\xi\mathbb E\left(W_{n+4,\xi}^{-1}\right)+n\mathbb E\left(W^{-1}_{n+2,\xi}\right),\label{eqn_recurrence1}\\
\mathbb E\left(W^{-1}_{n,\xi}\right)&=\xi\mathbb E\left(W_{n+4,\xi}^{-2}\right)+n\mathbb E\left(W^{-2}_{n+2,\xi}\right),\label{eqn_recurrence2}
\end{align}
for $n\geq 3$. Thus,
\begin{align*}
\mathbb E\left(W_{n+4,\xi}^{-1}\right)&=\frac{1}{\xi}-\frac{n}{\xi}\mathbb E\left(W^{-1}_{n+2,\xi}\right)\\
&\leq \frac{1}{\xi}-\frac{n}{\xi}\frac{1}{n+\xi}\\
&=\frac{1}{n+\xi}.
\end{align*}
Replacing $n$ by $n-4$ proves \eqref{eqn_inversechisq1}. On the other hand, rearranging \eqref{eqn_recurrence1}, we get
\begin{align*}
\mathbb E\left(W^{-1}_{n+2,\xi}\right)&=\frac{1}{n}-\frac{\xi}{n}\mathbb E\left(W^{-1}_{n+4,\xi}\right)\\
&\leq \frac{1}{n}-\frac{\xi}{n}\frac{1}{n+\xi+2}\\
&=\frac{n+2}{n(n+\xi+2)}.
\end{align*}
Now using \eqref{eqn_recurrence2}, we have
\begin{align*}
\mathbb E\left(W_{n+4,\xi}^{-2}\right)&=\frac{1}{\xi}\mathbb E\left(W^{-1}_{n,\xi}\right)-\frac{n}{\xi}\mathbb E\left(W^{-2}_{n+2,\xi}\right)\\
&\leq\frac{n}{\xi(n-2)(n+\xi)}-\frac{n}{\xi(n+\xi)(n+\xi-2)}\\
&=\frac{n}{(n-2)(n+\xi)(n+\xi-2)}.
\end{align*}
Replacing $n$ by $n-4$ proves \eqref{eqn_inversechisq2}.
\end{proof}

\end{appendix}

\end{document}